\documentclass[a4paper,11pt,oneside,reqno]{amsart}

\usepackage{setspace}
\setdisplayskipstretch{}
\usepackage[heightrounded]{geometry}
\usepackage{changepage}

\usepackage[english]{babel}
\usepackage{csquotes}

\usepackage{lmodern}
\usepackage{amsmath,amssymb}
\usepackage{mathtools}
\usepackage{nicematrix}
\usepackage{bm}
\usepackage{stmaryrd}
\usepackage{caption}
\usepackage{subcaption}
\usepackage{keytheorems} 

\usepackage{xcolor}
\usepackage{enumitem}
\usepackage{xpatch}
\usepackage{tikz-cd}
\usepackage{booktabs}

\usepackage{algorithm}
\usepackage[italicComments=false
           ,spaceRequire=false]
           {algpseudocodex}

\usepackage[textsize=tiny]{todonotes}

\definecolor{linkblue}{RGB}{1,1,190}
\definecolor{citered}{RGB}{190,1,1}
\definecolor{recursionblue}{RGB}{70,130,180}

\usepackage[linkcolor=linkblue
           ,urlcolor=linkblue
           ,citecolor=citered
           ,colorlinks
           ,bookmarksopen=true,
           ]{hyperref}

\makeatletter
  \AtBeginDocument{
    \hypersetup{
      pdftitle  = {\@title},
    }
  }
\makeatother
  
\usepackage[capitalize]{cleveref}

\newkeytheorem{theorem}[numberwithin=section]
\newkeytheorem{
  definition,
  lemma,
  proposition,
  corollary,
  question}
[numberlike=theorem]

\newkeytheorem{
  remark,
  remarks,
  example,
  examples
}[numberlike=theorem,style=remark]

\newkeytheorem{
  ouralgorithm
}[numberlike=theorem,style=remark,name=Algorithm]

\newcounter{Halgorithmic}
\AtBeginEnvironment{algorithmic}{\stepcounter{Halgorithmic}}
\ExpandArgs{c}\renewcommand{theHALG@line}{\arabic{Halgorithmic}.\arabic{ALG@line}}

\setlist[enumerate,1]{label=\textup{(\arabic*)}, ref=\textup{(}\arabic*\textup{)},
  itemsep=0.5em plus 0.15em minus 0.05em,
  topsep=0.5em plus 0.15em minus 0.05em,
  leftmargin=0.75cm}
\setlist[enumerate,2]{label=\textup{(\roman*)}, ref=\textup{(}\roman*\textup{)},
  itemsep=0.5em plus 0.15em minus 0.05em,
  topsep=0.5em plus 0.15em minus 0.05em}
\setlist[itemize, 1]{itemsep=0.5em plus 0.15em minus 0.05em,
  topsep=0.5em plus 0.15em minus 0.05em, leftmargin=0.75cm}

\newlist{equivenumerate}{enumerate}{1}
\setlist[equivenumerate,1]{%
  label=\textup{(\alph*)},
  ref=\textup{(}\alph*\textup{)},
  itemsep=0.5em plus 0.15em minus 0.05em,
  topsep=0.5em plus 0.15em minus 0.05em,
  leftmargin=0.75cm
}

\newlist{propenumerate}{enumerate}{1}
\setlist[propenumerate,1]{%
  label=\textup{(\roman*)},
  ref=\textup{(}\roman*\textup{)},
  itemsep=0.5em plus 0.15em minus 0.05em,
  topsep=0.5em plus 0.15em minus 0.05em,
  leftmargin=0.75cm
}

\newlist{proofenumerate}{enumerate}{1}
\setlist[proofenumerate,1]{%
  itemsep=0.5em plus 0.15em minus 0.05em,
  topsep=0.5em plus 0.15em minus 0.05em,
  wide, labelindent=0pt
}

\xpatchcmd{\paragraph}{\normalfont}{{\normalfont\bfseries}}{}{}

\usepackage[backend=biber,
            style=alphabetic,
            giveninits=true,
            url=false,
            isbn=false,
            maxnames=99,
           ]{biblatex}
\renewbibmacro{in:}{}
\DeclareFieldFormat[article,inbook,incollection,inproceedings,patent,thesis,unpublished]{title}{{#1\isdot}}
\DeclareFieldFormat[article,inbook,incollection,inproceedings,patent,thesis,unpublished]{note}{\textit{#1\isdot}}
\DeclareFieldFormat[article,inproceedings,incollection]{pages}{#1}
\DeclareFieldFormat{postnote}{#1} 

\renewbibmacro*{issue+date}{%
  \iffieldundef{year}{}{ 
  \printtext[parens]{%
    \printfield{issue}%
    \setunit*{\addspace}%
    \usebibmacro{date}}%
  \newunit}}

\AtBeginBibliography{\small}

\usepackage{microtype}

\newcommand{\defit}[1]{\textsf{#1}}

\newcommand{\bF}{\mathbb F}

\newcommand{\bZ}{\mathbb Z}
\newcommand{\bP}{\mathbb P}
\newcommand{\bQ}{\mathbb Q}
\newcommand{\bR}{\mathbb R}
\newcommand{\bC}{\mathbb C}

\newcommand{\cF}{\mathcal F}

\newcommand{\htop}{\mathfrak X}

\newcommand{\vprod}{\sideset{}{\mkern -2mu\raisebox{0.2em}{$\vphantom{\prod}^v$}} \prod}

\newcommand{\val}{\mathsf{v}}
\DeclareMathOperator{\dv}{div} 

\newcommand{\algc}[1]{\overline{#1}}

\renewcommand{\vec}[1]{\bm{#1}}

\DeclareMathOperator{\chr}{char}
\DeclareMathOperator{\Pic}{Pic}
\DeclareMathOperator{\Cl}{Cl}
\DeclareMathOperator{\Div}{Div}
\DeclareMathOperator{\Princ}{Princ}
\DeclareMathOperator{\Spec}{Spec}
\DeclareMathOperator{\height}{ht}
\DeclareMathOperator{\Coker}{Coker}
\DeclareMathOperator{\im}{im}
\DeclareMathOperator{\Partner}{Partner}
\DeclareMathOperator{\Trop}{Trop}

\DeclarePairedDelimiter{\card}{\lvert}{\rvert}

\title[Factoriality and Class Groups of Upper Cluster Algebras and FLIRs]{Factoriality and Class Groups of Upper Cluster Algebras and Finite Laurent Intersection Rings: A Computational Approach}

\author{Mara Pompili}
\address{University of Graz, Department of Mathematics and Scientific Computing, NAWI Graz, Heinrichstrasse 36, 8010 Graz, Austria}
\email{mara.pompili@uni-graz.at}

\author{Daniel Smertnig}
\address{Faculty of Mathematics and Physics (FMF)\\
  University of Ljubljana
  and Institute of Mathematics, Physics and Mechanics (IMFM)\\
  Jadranska ulica 21\\
  1000 Ljubljana, Slovenia}
\email{daniel.smertnig@fmf.uni-lj.si}


\begin{document}

\begin{abstract}
  \begin{singlespace}
    We study factoriality and the class groups of locally acyclic cluster algebras.
    To do so, we introduce a new class of rings called finite Laurent intersection rings (FLIRs), which includes locally acyclic cluster algebras, full-rank upper cluster algebras, and certain generalized upper cluster algebras and Laurent phenomenon algebras.
    Our main results are algorithms to compute the class group of an explicit FLIR, to determine factoriality, and to compute all factorizations of a given element.
    The algorithms are based on multivariate polynomial factorizations, avoiding computationally expensive Gröbner basis calculations.
  \end{singlespace}
\end{abstract}

\date{}

\maketitle

\vspace*{-1cm}
\setcounter{tocdepth}{1}
\begin{center}
\parbox{0.7\textwidth}
{\small\tableofcontents}
\end{center}
\setcounter{tocdepth}{2}

\section{Introduction}

A basic ring-theoretic question about a cluster algebra or upper cluster algebra is whether it is factorial, that is, whether every nonzero element admits a unique factorization into atoms (i.e., irreducible elements), up to order and units.
Upper cluster algebras are often noetherian and integrally closed, or are at least Krull domains, in which case they are factorial if and only if their (divisor) class group is trivial.
The more general question is therefore to determine the class group.
This gives a complete picture of the factorization theory even in the non-factorial case.

Geiss, Leclerc, and Schröer first studied the factoriality of cluster algebras, and gave sufficient conditions \cite{GLS13}.
Garcia Elsener, Lampe, and the second author characterized the factoriality of cluster algebras with acyclic initial seed and proved an elementary formula for the class group in terms of the initial exchange matrix \cite{GELS19}.
The determination of the class groups of finite-type and affine cluster algebras follows as an easy consequence \cite[Corollaries 5.16 and 5.17]{GELS19}.

Acyclic cluster algebras coincide with their upper cluster algebras \cite{BFZ05,M14}.
In general, upper cluster algebras are the more natural objects from a ring-theoretic and geometric perspective.
For instance, they arise as the ring of global sections on $\mathcal A$-cluster varieties \cite[Remark 2.9]{GHK15}.
Shifting the focus to upper cluster algebras, Cao, Keller, and Qin \cite{CKQ22} proved that a full-rank upper cluster algebra is factorial if and only if its initial exchange polynomials are irreducible.
The first author extended this to the determination of the class group \cite{Pom25a}, and obtained results for generalized cluster algebras and Laurent phenomenon algebras \cite[\S4 and \S5]{Pom25b}.

These algebras have in common that, starting from an initial seed, it suffices to consider a few particular mutations to discover the complete height-one spectrum of the algebra: in the case of full-rank upper cluster algebras, the Starfish Lemma \cite[Corollary 1.9]{BFZ05} shows that it is sufficient to mutate once in each direction from the initial seed.
For acyclic cluster algebras, a sequence of well-chosen freezings and mutations achieves a similarly explicit understanding of the height-one spectrum \cite[\S4]{GELS19} \cite{M14}.

Locally acyclic cluster algebras \cite{M13} form a large class of cluster algebras which coincide with their upper cluster algebras \cite{M14}.
They include acyclic cluster algebras, cluster algebras arising from marked surfaces with at least two marked points on each boundary component \cite[Theorem 10.6]{M13}, coordinate rings of Grassmannians \cite{MS16}, and more generally, of open positroid and braid varieties \cite{galashin-lam23,CGGSS25}.

If $A$ is a locally acyclic cluster algebra, then $\Spec(A)$ still has a cluster cover by finitely many isolated cluster localizations.
However, if the exchange matrix does not have full rank, the cover could involve a complex sequence of mutations.
It would be too much to expect a simple formula for the rank of the class group in terms of the initial seed.

We overcome this problem in the present paper by taking an algorithmic approach: instead of an explicit formula, we give an algorithm that, given as input a locally acyclic cluster algebra, computes its class group.
This is comparable to the situation of rings of integers in number fields, where no general formula for the class group is known, but the class group is computable for any given ring of integers.

Once the class group is computable, it is then trivially decidable whether the algebra is factorial.
But under this computational point of view an additional natural algorithmic question arises: given an element of the algebra, compute all its factorization into atoms.
We show that this is also possible.

To achieve such computational results, the ground ring itself needs to be computable in a suitable sense, and a cluster cover by isolated (or acyclic) cluster localizations must be explicitly known.
The latter condition, while not automatic, is often achievable, because many cluster algebras are proven to be locally acyclic by the Banff algorithm \cite[\S5]{M13}, which explicitly constructs such a cover.
For instance, this is the case for cluster algebras of surfaces with at least two marked points on each boundary component.

A simplified version of our main result is the following.
An exchange matrix satisfies the Banff property if the Banff algorithm successfully terminates when started with that matrix.
Computable structures are discussed in \cref{s:computable-domains}.

\begin{theorem} \label{t:main-banff}
  Let $D$ be a computable Krull domain \textup(such as $\bQ$, $\algc{\bQ}$, a finite field, or $\bZ$\textup), and $\bP$ a computable semifield \textup(such as a tropical semifield\textup). 
  Let $(\vec x, \vec y, B)$ be a seed with the exchange matrix $B$ satisfying the Banff property, and let $A=A(\vec x, \vec y, B)$ be the corresponding cluster algebra over $D$ with coefficients in $\bP$.
  Then there exist algorithms to:
  \begin{enumerate}
  \item compute the class group $\Cl(A)$;
  \item decide whether $A$ is factorial; and
  \item given $0\ne a \in A$, compute all factorizations of $a$ into atoms of $A$.
  \end{enumerate}
\end{theorem}

In fact, we first introduce a new class of rings, called \emph{finite Laurent intersection rings} (\emph{FLIRs}) in \cref{d:lir,d:flir}, which captures the relevant key properties of upper cluster algebras in height one.
FLIRs have a very simple definition while being versatile enough to include various generalizations of upper cluster algebras, such as interesting special cases of generalized cluster algebras and Laurent phenomenon algebras.
Since FLIRs over ground rings that are Krull domains are themselves Krull domains (\cref{p:lir-krull}), they provide a natural setting to study the factorization theory of such algebras in unified way.

Unlike locally acyclic algebras, the more general class of FLIRs may have \emph{deep points} in the sense of \cite{CGSS24,BM25}.
However, these deep points have height at least two, and hence do not impair the study of factorizations (\cref{r:krull-upper-cluster-algebras}).

As in \cite{GELS19,Pom25a}, our methods are based on tools from multiplicative ideal theory, in particular, Krull domains and Nagata's Theorem (see \cref{ssec:krull} below).
The recently introduced valuation pairings \cite{CKQ22} can be understood in terms of discrete valuations in this setting \cite[\S5]{Pom25a}.
We introduce \emph{computational Krull domains} in \cref{d:computable-domain}, giving a rigorous definition of what it means to be able to perform the basic operations in a Krull domain algorithmically.
Aside from the arithmetic operations, this includes computations with divisors and class groups.
Of course, only some ground rings, such as $\bZ$, $\bQ$, $\algc{\bQ}$, number fields and their rings of integers, and finite fields are really of interest. 
However, the setting of computable Krull domains allows us to deal with various ground rings cleanly and uniformly.

In the setting of computational Krull domains, our main results are that all factorizations of an element are computable (\cref{t:compute-factorizations}) and that explicit FLIRs over computable Krull domains, with a splitting algorithm for the field of fractions, are again computable Krull domains (\cref{t:flir-computable}).

After this groundwork, it is easy to observe that the Banff algorithm shows that Banff cluster algebras are explicit FLIRs (\cref{t:banff-explicit-FLIR}) and to thereby obtain \cref{t:main-banff}.

The framework of FLIRs also straightforwardly applies to recover factorization results on full-rank upper cluster algebras \cite{CKQ22,Pom25a} and acyclic cluster algebras \cite{GELS19}.
We do not pursue this for full-rank upper cluster algebras, as it does not offer much over \cite{Pom25a}.
For acyclic cluster algebras, we give a simpler and more conceptual proof of the most technical step of \cite{GELS19} and answer \cite[Question 4.8]{GELS19} in \cref{ssec:acyclic-primes}.

A key feature of our algorithms is that they are reasonably practical: they completely avoid the use of Gröbner bases, which can quickly become impractical as the number of variables or the size of the cluster cover grow.
Instead, the algorithms are based on the Laurent phenomenon and factoring multivariate polynomials, for which practical algorithms exist \cite{kaltofen92,kaltofen03}.
In fact, since Laurent polynomial rings are special cases of upper cluster algebras, the complexity of factoring multivariate polynomials also provides a natural lower bound for the complexity of factoring elements in upper cluster algebras.
We implemented a prototype version of our algorithms using the computer algebra system SageMath \cite{sagemath}.

In \cref{ssec:fin-pres}, we also describe an alternative approach, in which the Banff algorithm is adapted to compute finite presentations of Banff cluster algebras (using Gröbner bases), and the class group is determined using primary decompositions.
The computation of presentations of upper cluster algebras was first considered by Matherne and Muller \cite{MM15}.
In \cref{ssec:examples} we compare the polynomial factorization approach with the Gröbner basis approach on a family of examples.
The results strongly indicate that the asymptotic performance advantage of our Laurent-phenomenon-based algorithms indeed translates into a practical performance advantage that greatly extends the class of cluster algebras for which such computations are even feasible (\cref{tab:comparison}).

\smallskip
\paragraph{Acknowledgements} Pompili was partially supported by the Austrian Science Fund (FWF), project 10.55776/DOC-183-N. Smertnig was supported by the Slovenian Research and Innovation Agency (ARIS) program P1-0288 and grant J1-60025.

\section{Preliminaries} \label{s:prelim}

By $[1,n]$ we denote the discrete interval $\{\, a \in \bZ : 1 \le a \le n\,\}$.
A \defit{domain} $D$ is a commutative ring with unity in which zero is the only zero-divisor.
By $D^\times$ we denote its group of units, and by $\htop(D)$ the set of height-one prime ideals.

It will be convenient to use multi-index notation: given a tuple of variables $\vec x = (x_1, \ldots, x_n)$, we denote by $D[\vec x]$ the polynomial ring $D[x_1, \ldots, x_n]$ and by $D[\vec x^{\pm 1}]$ the corresponding Laurent polynomial ring.
Similarly, if $K$ is a field, then $K(\vec x)$ denotes the field of  rational functions.
If $\vec m=(m_1,\dots,m_n) \in \bZ^n$, we write $\vec x^{\vec m} \coloneqq x_1^{m_1} \cdots x_n^{m_n}$.

\subsection{Factorization}

We recall some basic notions from factorization theory; a standard reference is \cite{GH06}, recent surveys are \cite{smertnig16,geroldinger-zhong20,anderson-gotti21,coykendall-gotti25,geroldinger-kim-loper26}, for a category-theoretical point of view see \cite{campanini-cossu-tringali25}.
Recent generalizations of these notions are studied in \cite{anderson-frazier11,tringali22,cossu-tringali24,cossu25}.
Let $D$ be a domain.

A nonzero nonunit $u\in D$ is an \defit{atom} (equivalently, an \defit{irreducible element}) if $u = ab$ with $a$,~$b \in D$ implies $a\in D^\times$ or $b\in D^\times$.
A nonzero nonunit $p\in D$ is a \defit{prime element} if $p \mid ab$ implies $p \mid a$ or $p\mid b$ for all $a$,~$b\in D$.
Every prime element is an atom, but in general not conversely.

The domain $D$ is \defit{atomic} if every nonzero nonunit is a product of atoms; it is \defit{factorial} if every nonzero element is a product of prime elements. 
The following is standard.
\begin{theorem}
  For an atomic domain, the following statements are equivalent.
  \begin{equivenumerate}
    \item The domain $D$ is factorial.
    \item Every atom of $D$ is prime.
    \item Any two factorizations of the same element into atoms differ only in the order and in the associates of the atoms.
  \end{equivenumerate}
\end{theorem}

Here, two elements $a$,~$b \in D$ are \defit{associates} if $aD=bD$, equivalently, if there exists a unit $\varepsilon \in D^\times$ such that $a=\varepsilon b$.
Being atomic is in general a much weaker property than being factorial.

The domains under consideration in this paper will always be \defit{finite factorization domains \textup(FFDs)}: each nonzero element has only finitely many distinct factorizations into atoms, up to order and associativity of the factors \cite[Chapter 1.5]{GH06} \cite{anderson-gotti21}.
In such domains, it makes sense to ask to compute all factorizations of a given element.

\subsection{Krull Domains and Class Groups} \label{ssec:krull}

Krull domains are a central class of rings studied in multiplicative ideal theory and factorization theory.
For our purposes, the classical treatment in \cite{FOSSUM} is sufficient.
Other classical references are \cite{gilmer92} \cite[Chapter~VII]{bourbaki-CA72}.
Modern treatments can be found in \cite{GH06,elliott19,wang-kim24,halter-koch25}.

Krull domains have many equivalent characterizations, the following will be most useful to us.
Let $D$ be a domain with field of fractions $K$.

\begin{definition} \label{d:krull-domain}
  The domain $D$ is a \defit{Krull domain} if there exists a family of discrete valuation rings \textup(DVRs\textup) $(V_i)_{i \in I}$ with $K$ as field of fractions such that
  \begin{propenumerate}
    \item $D = \bigcap_{i \in I} V_i$, and
    \item for every $0 \ne a \in D$, the set $\{\, i \in I : a \notin V_i^\times \,\}$ is finite.
  \end{propenumerate}
\end{definition} 

The second property is known as the \defit{finite character property}.

If $P \in \htop(D)$ is a height-one prime ideal of a Krull domain $D$, then the localization $D_{P}$ is a DVR \cite[Proposition 1.9]{FOSSUM}, and
\[
D = \bigcap_{P \in \htop(D)} D_P.
\]

Subintersections of Krull domains, finite intersections of Krull domains with the same field of fractions, and localizations of Krull domains are easily seen to be again Krull domains \cite[Chapter~1]{FOSSUM}.
Fields and factorial domains are Krull domains.

\subsubsection{\textup(Laurent\textup) Polynomial Rings} \label{sssec:div-poly}
A polynomial ring $D[\vec x]$ is a Krull domain if and only if the base ring $D$ is a Krull domain, and the same is true for $D[\vec x^{\pm 1}]$.
The height-one spectrum is in bijection with the disjoint union $\htop(D) \sqcup \htop(K[\vec x])$, respectively, with $\htop(D) \sqcup \htop(K[\vec x^{\pm 1}])$.
This is not hard to check, and also follows from analogous statements holding more generally for Krull monoid algebras \cite[\S15]{gilmer84}.

More specifically, if $P \in \htop(D[\vec x])$, then there are two possibilities.
\begin{itemize}
  \item If $P \cap D \ne 0$, then $P \cap D \in \htop(D)$ and $P = (P \cap D)[\vec x]$. The associated valuation $\val_P\colon K(\vec x) \to \bZ \cup \{\infty\}$ is on $D[\vec x]$ defined by
  \[
    \val_P\Big( \sum_{\vec m} a_{\vec m} \vec x^{\vec m} \Big) = \min\{\, \val_{P \cap D}(a_{\vec m}) : \vec m \,\}.
  \]
  \item If $P \cap D = 0$, then $P$ localizes to a height-one prime ideal of $K[\vec x]$ and $P = P K[\vec x] \cap D[\vec x]$.
  Since $K[\vec x]$ is factorial, there exists an irreducible polynomial $p \in K[\vec x]$ with $P K[\vec x] = p K[\vec x]$, and $\val_P(f) = \val_p(f)$ is the multiplicity of $p$ in $f$ for $f \in D[\vec x]$.
\end{itemize}

For the Laurent polynomial ring $D[\vec x^{\pm 1}]$, the same argument shows that $\htop(D[\vec x^{\pm 1}])$ is in bijection with the disjoint union of $\htop(D)$ and $\htop(K[\vec x^{\pm 1}])$, that is, in contrast to $D[\vec x]$, the prime ideals $(x_1)$, \dots,~$(x_n)$ do not occur.

\subsubsection{Divisorial Ideals and Divisors} \label{sssec:divisors}
Let $D$ be a Krull domain.
A \defit{fractional ideal} of $D$ is a nonzero\footnote{Sometimes the zero module is allowed, but we exclude it here since this is unnecessary and avoids us having to exclude it all the time. Correspondingly, for us $\cF_v(D)$ will be the set of all nonzero divisorial fractional ideals, whereas other texts commonly write $\cF_v(D)^\times$ to exclude $0$.} $D$-submodule $I$ of $K$ for which there exists $x\in K^\times$ with $xI\subseteq D$.
For a fractional ideal $I$, let
\[
I^{-1}\coloneqq (D\colon\! I) \coloneqq \{\, x \in K : xI \subseteq D \,\} \qquad \text{and} \qquad I_v \coloneqq (I^{-1})^{-1}.
\]
Then $I$ is \defit{divisorial} if $I=I_v$.

The set of nonzero divisorial prime ideals of $D$ is precisely the set of height-one prime ideals $\htop(D)$ \cite[Theorem 3.12]{FOSSUM}.
With the operation of $v$-multiplication, defined by $I \cdot_v J \coloneqq (IJ)_v$, the set of divisorial fractional ideals $\mathcal{F}_v(D)$ forms a free abelian group with basis $\htop(D)$ and identity element $D$  \cite[Corollary 3.14]{FOSSUM}.
Explicitly, for $I \in \cF_v(D)$, there is a unique presentation
\begin{equation} \label{eq:prime-factorization}
  I = \vprod_{P \in \htop(D)} P^{\val_P(I)},
\end{equation}
where the symbol $\vprod$ denotes the $v$-product, where $\val_P(I) \in \bZ$, and where all but finitely many factors are trivial, that is, equal to $D$.

Divisorial closure commutes with localization, explicitly $S^{-1} I_v = (S^{-1} I)_v$ for $I \in \cF_v(D)$ and a multiplicative set $S \subseteq D$ \cite[Corollary 5.5]{FOSSUM}.
Localizing at $P \in \htop(D)$, one sees that $\val_P(I) = \min\{\, \val_P(a) : a \in I \,\}$.

\cref{eq:prime-factorization} motivates the notion of \defit{divisors}.
We call $\htop(D)$ the set of \defit{prime divisors} of $D$.
A \defit{divisor} $E$ is then a formal $\bZ$-linear combination of prime divisors, that is,
\[
  E = \sum_{P \in \htop(D)} n_P P, \quad n_P \in \bZ, \text{ almost all $0$}.
\]
The set of divisors, denoted by $\Div(D)$, is a free abelian group with basis $\htop(D)$.
For $a \in K^\times$, the \defit{principal divisor} associated to $a$ is
\[
  \dv(a) \coloneqq \dv_D(a) \coloneqq \sum_{P \in \htop(D)} \val_P(a) P.
\]
Let $\Princ(D)$ denote the subgroup of principal divisors.

The isomorphism $\mathcal{F}_v(D) \cong \Div(D)$ of free abelian groups maps the $v$-product of divisorial fractional ideals to addition in $\Div(D)$.
Additionally \cite[Proposition 5.9]{FOSSUM},
\[
  \dv\big( (I+J)_v \big) = \min\{ \dv(I), \dv(J) \} = \sum_{P \in \htop(D)} \min\{ \val_P(I), \val_P(J) \} P.
\]

While the isomorphism is vacuous algebraically, it will be important for the computational considerations later on: 
a divisorial fractional ideal is most naturally represented by a (finite) set of generators: $I= (a_1,\dots,a_n)_v$. A divisor is most naturally represented as a formal linear combination of prime divisors, that is, by giving the coefficients.
Passing between the two representations is computationally a non-trivial task, making it important to distinguish between the two different representations.

\subsubsection{Class Groups} 
The class group is the main invariant of a Krull domain, with its computation being a central task of the present paper.
\begin{definition} \label{d:class-group}
  The \defit{\textup(divisor\textup) class group} of a Krull domain $D$ is
  \[
  \Cl(D) \coloneqq \cF_v(D) / \{\, aD : a \in K^\times \} \cong \Div(D) / \Princ(D).
  \]
\end{definition}

\begin{remark}
  A noetherian domain is a Krull domain if and only if it is integrally closed (that is, normal) \cite[Theorem 2.10.2]{GH06}.
  In this case $\Div(D)$ coincides with the group of Weil divisors on $\Spec(D)$, and the class group $\Cl(D)$ is the corresponding Weil divisor class group.
\end{remark}

The importance of Krull domains and their class groups for factorization theory is highlighted by the following classical result.

\begin{theorem}[{\cite[Proposition 6.1]{FOSSUM}}]
  A domain $D$ is factorial if and only if it is a Krull domain with trivial class group.
\end{theorem}

Using the unique factorization in $\cF_v(D)$, it is easy to see that Krull domains are FFDs. 
Knowledge of the class group together with the distribution of prime divisors among classes allows one to describe the arithmetic of $D$ in great detail: modulo invertible elements, the multiplicative monoid of $D$ is uniquely determined up to isomorphism by these data \cite[Theorem 2.5.4.3]{GH06}.
Using transfer homomorphisms to monoids of zero-sum sequences one can study the factorization theory of $D$ \cite{GH06} \cite[\S5]{geroldinger-zhong20} \cite[\S3]{geroldinger16}.

We need results on the behavior of class groups under localization and polynomial extensions.
There is a homomorphism
\[
  \cF_v(D) \to \cF_v(S^{-1} D), \quad I \mapsto S^{-1} I,
\]
that induces a homomorphism $\Cl(D) \to \Cl(S^{-1} D)$.
Nagata's Theorem describes the kernel.

\begin{theorem}[{Nagata's Theorem \cite[Corollary 7.2]{FOSSUM}}]\label{t:nagata}
  Let $D$ be a Krull domain and let $S \subseteq D\setminus\{0\}$ be a multiplicative subset.
  Then there is an exact sequence
  \[
  \begin{tikzcd}
  0 \ar[r] & H \ar[r] & \Cl(D) \ar[r] & \Cl(S^{-1} D) \ar[r] & 0,
  \end{tikzcd}
  \]
  where $H$ is the subgroup of $\Cl(D)$ generated by those $P \in \htop(D)$ with $P \cap S\ne \emptyset$.
\end{theorem}

\begin{corollary}[{\cite[Chapter 8]{FOSSUM}} or \cite[\S16]{gilmer84}] \label{c:cl-polynomial}
  If $D$ is a Krull domain, then 
  \[
  \Cl(D)\cong \Cl(D[\vec{x}])\cong \Cl(D[\vec{x}^{\pm1}]),
  \]
  with the first isomorphism induced by extension of divisorial fractional ideals: $[I] \mapsto \big[I D[\vec x]\big]$ for $I \in \cF_v(D)$, and the second one by localization.
\end{corollary}

\subsection{Cluster Algebras and Seeds}
We recall cluster algebras and upper cluster algebras, standard references are \cite{FZ02,FZ03_2,BFZ05,Wil14,FWZ,FWZ2,FWZ3}.

(Upper) cluster algebras are defined over a ground ring $D$.
In the context of combinatorics, the ground ring often does not matter too much and is commonly taken to be $\bQ$, $\bR$, or $\bZ$.
By contrast, the ring-theoretic properties we investigate depend heavily on the choice of $D$.
We allow arbitrary domains $D$ (including those of positive characteristic) as ground rings, but for the factorization-theoric results we will have to restrict to Krull domains $D$ later on.
This class is nevertheless wide, as it includes all fields and factorial domains, such as the integers.

\subsubsection{Seeds and Mutations} \label{sssec:seed-and-mutations}
Let $(\bP,\oplus,\cdot)$ be a \defit{semifield}: a torsion-free\footnote{To ensure that the group algebra $D\bP$ is a domain.} abelian group $(\bP,\cdot)$ together with an additional commutative semigroup structure $\oplus$ such that the product distributes over $\oplus$.
The main example is the \defit{tropical semifield} $\Trop(y_1,\dots,y_m)$, which is generated by variables $y_1$, \dots,~$y_m$ with the auxiliary addition
\[
  \prod_{i=1}^m y_i^{a_i} \oplus \prod_{i=1}^m y_i^{b_i} \coloneqq \prod_{i=1}^m y_i^{\min(a_i,b_i)}.
\]
In this case, the group algebra $D\bP$ is isomorphic to the Laurent polynomial ring $D[\vec y^{\pm 1}]$.

Let $K$ be the field of fractions of $D\bP$, and let $F$ be the field of rational functions in $n$ variables over $K$.
A \defit{seed} of rank $n$ is a triple $(\vec x, \vec y, B)$ consisting of
\begin{itemize}
  \item a \defit{cluster}: a tuple $\vec x = (x_1,\dots,x_n)$ whose elements generate the field $F \supseteq K$.
  \item \defit{coefficients}: a tuple $\vec y = (y_1,\dots,y_n) \in \bP^n$, and
  \item an \defit{exchange matrix}: an $n\times n$ skew-symmetrizable integer matrix $B=(b_{ij})$, that is, there exist positive integers $d_1$, \dots, $d_n$ such that $d_i b_{ij} = -d_j b_{ji}$ for all $i$,~$j \in [1,n]$.
\end{itemize}
The elements of a cluster are called \defit{cluster variables}.

To allow $\chr D=2$, we make one extra assumption on the seed (similar to \cite{BMRS15}): if $\chr D=2$ and the $k$-th row of $B$ is zero, then we must have $y_k \ne 1$.
This will ensure that mutation is well-defined.

A seed $(\vec x, \vec y, B)$ can be \defit{mutated} in any direction $k \in [1,n]$ to produce a new seed $(\mu_k(\vec x), \mu_k(\vec y), \mu_k(B))$ as follows.
\begin{itemize}
  \item The tuple $\mu_k(\vec x)$ arises from $\vec x$, by replacing $x_k$ by $x_k'$, defined by
  \[
    x_k x_k' = \frac{y_k}{y_k \oplus 1} \prod_{b_{ik}>0} x_i^{b_{ik}} + \frac{1}{y_k \oplus 1}\prod_{b_{ik}<0} x_i^{-b_{ik}},
  \]
  and leaving all other variables unchanged.
  \item Set $\mu_k(\vec y)=(y_1',\dots,y_n')$ with
  \[
    y_j' = \begin{cases}
      y_k^{-1} & \text{if } j=k, \\
      y_j (y_k \oplus 1)^{-b_{kj}} & \text{if $j \ne k$ and $b_{kj} \le 0$,} \\
      y_j (y_k^{-1} \oplus 1)^{-b_{kj}} & \text{if $j \ne k$ and $b_{kj} \ge 0$.}
    \end{cases}
  \]
  \item If $B=(b_{ij})$, then $\mu_k(B)=(b_{ij}')$ with
  \[
    b_{ij}' = \begin{cases}
      -b_{ij} & \text{if $i=k$ or $j=k$}, \\
      b_{ij} + b_{ik}b_{kj} & \text{if $b_{ik}>0$ and $b_{kj}>0$}, \\
      b_{ij} - b_{ik}b_{kj} & \text{if $b_{ik}<0$ and $b_{kj}<0$}, \\
      b_{ij} & \text{otherwise}.
    \end{cases}
  \]
\end{itemize}

The polynomials $f_k = x_kx_k'\in D\bP[\vec{x}]$ are the \defit{exchange polynomials} associated to the seed $(\vec{x},\vec{y},B)$.

Mutation of $\vec x$ depends on the ground ring $D$, as the addition is carried out over $D$, while mutation of $\vec y$ and $B$ does not depend on $D$.
Since the ground ring is usually fixed, it is not reflected in the notation.
If we want to emphasize the dependence on $D$, we write $(\vec x, \vec y, B; D)$ for the seed.

\begin{remark}
  If $\chr D=2$, then the assumption on $B$ and $y_k$ ensures $x_k x_k' \ne 0$, thereby ensuring that $\mu_k(\vec x)$ is again a tuple of generators for $F$.
  It is easy to check that the assumption is retained under mutation.
\end{remark}

Two seeds are \defit{mutation-equivalent} if one can be obtained from the other by a sequence of mutations.

\subsubsection{Cluster Algebras}

\begin{definition}
  Let $D$ be a domain, let $\bP$ be a semifield, and let $F$ be the field of rational functions in $n$ variables over the field of fractions of $D\bP$.
  Let $(\vec x, \vec y, B)$ be a seed.
  \begin{enumerate}
    \item The \defit{cluster algebra} $A\coloneqq A(B) \coloneqq A(\vec x,\vec y,B;D)$ associated with the initial seed $(\vec x, \vec y, B)$ is the $D \bP$-subalgebra of $F$ generated by all cluster variables in all seeds mutation-equivalent to $(\vec x, \vec y, B)$.
    \item The \defit{upper cluster algebra} $U \coloneqq U(B) \coloneqq U(\vec x, \vec y, B;D)$ associated with the initial seed $(\vec{x}, \vec y, B)$ is
    \[
     U(\vec x, \vec y, B;D)=\bigcap_{\widetilde{\vec x}} D\bP[\widetilde{\vec x}^{\pm 1}],
    \]
    where the intersection is over all clusters $\widetilde{\vec x}$ in seeds mutation-equivalent to $(\vec x, \vec y, B)$.
  \end{enumerate}
\end{definition}

\begin{theorem}[Laurent phenomenon {\cite[Theorem 3.1]{FZ02}}]
  Every cluster variable is a Laurent polynomial in the variables of any given cluster.
  In other words, there is an inclusion $A \subseteq D\bP[\vec x^{\pm 1}]$ for any cluster $\vec x$.
\end{theorem}

This is typically proven for $D=\bZ$ and easily carries over to any domain of characteristic $0$: in this case $\bZ P \subseteq D \bP$ and all mutation operations occur in the subring $\bZ \bP$.
We sketch that the same approach works in arbitrary characteristic if some care is taken.

\begin{proof}[Proof sketch for arbitrary domains $D$]
  Let $p = \chr D \ge 0$.
  There is a ring homomorphism
  \[
    \pi \colon \bZ \bP \to (\bZ/p\bZ) \bP \hookrightarrow D \bP,
  \]
  with kernel $p \bZ \bP$.
  This induces a homomorphism $\pi\colon \bZ \bP [\vec x^{\pm 1}] \to D \bP[\vec x^{\pm 1}] \subseteq F$.

  If we consider the cluster algebra defined over $\bZ$ by the seed $(\vec x, \vec y, B; \bZ)$, we get
  \[
  \pi(x_k x_k') = \frac{y_k}{y_k \oplus 1} \prod_{b_{ik}>0} \pi(x_i)^{b_{ik}} + \frac{1}{y_k \oplus 1}\prod_{b_{ik}<0} \pi(x_i)^{-b_{ik}}.
  \]
  Since we know that $x_k'$ is a Laurent polynomial in the $x_i$ over $\bZ \bP$ (by the Laurent phenomenon over $\bZ \bP$), the left side is equal to $\pi(x_k)\pi(x_k')$ (keep in mind that $\pi$ is only defined on $\bZ\bP[\vec x^{\pm 1}]$, not on its field of fractions).
  Inductively, we see that mutation and application of $\pi$ interchange, and the Laurent phenomenon over $\bZ \bP$ for $(\vec x, \vec y, B; \bZ)$ implies the Laurent phenomenon over $D \bP$ for $(\vec x, \vec y, B; D)$.
\end{proof}

The Laurent phenomenon shows that there is always an inclusion $A \subseteq U$ between a cluster algebra and its upper cluster algebra.
The equality
\[
  A=U
\]
is true in many important subclasses of cluster algebras, in particular in locally acyclic ones \cite{M14}.

\subsubsection{Locally Acyclic Cluster Algebras} \label{sssec:locally-acyclic}
Localizing a cluster algebra in general does not again yield a cluster algebra.
This led Muller to the notions of freezings and cluster localizations \cite{M13,M14}.
While the ground ring in \cite{M14} is $\bZ$, the arguments go through for arbitrary domains $D$.\footnote{In particular, this is true for the key \cite[Lemma 1]{M14}, which shows that if $A^\dagger=U^\dagger$ for a freezing, then $A^\dagger$ is a cluster localization of $A$, and for \cite[Proposition 3]{M14}, which shows $A=U$ for isolated cluster algebras.}

Let $(\vec x, \vec y, B)$ be a seed of rank $n$.
Its \defit{freezing} at a subset $S\coloneqq \{x_{i_1},\dots,x_{i_s}\} \subseteq \{ x_1,\dots, x_n \}$ of the cluster $\vec x$ is obtained as follows.
\begin{itemize}
\item Replace $\bP$ by
\[
\bP^\dagger \coloneqq \bP \oplus \Trop(x_{i_1},\dots,x_{i_s}) \cong \bP \oplus \bZ^{\card{I}}
\]
with
\[
  (p,x_{i_1}^{a_1} \cdots x_{i_s}^{a_s}) \oplus (q,x_{i_1}^{b_1} \cdots x_{i_s}^{b_s}) \coloneqq (p \oplus q, x_{i_1}^{\min(a_1,b_1)} \cdots x_{i_s}^{\min(a_s,b_s)}).
\]
\item Remove the entries of $S$ from $\vec x$ to obtain an $n-s$ tuple $\vec x^\dagger$,
\item Define $\vec y^\dagger = (y_j^\dagger : j \in [1,n] \setminus \{i_1,\dots,i_s\})$ (in order) with entries
\[
  y_j^\dagger = y_j x_{i_1}^{b_{i_1 j}} \cdots x_{i_s}^{b_{i_s j}}.
\]
\item Remove the rows and columns with indices $i_1$, \dots,~$i_s$ from $B$ to obtain $B^\dagger$.
\end{itemize}

The frozen seed $(\vec x^\dagger, \vec y^\dagger, B^\dagger)$ defines a cluster algebra $A[x_{i_1}^\dagger,\dots,x_{i_s}^\dagger]$ of rank $n-s$ called the \defit{freezing} of $A$ at $x_{i_1}$, $\dots$,~$x_{i_s}$, and similarly defines a freezing $U[x_{i_1}^\dagger,\dots,x_{i_s}^\dagger]$ of $U$.
The effect of freezing a variable $x_{i_j}$ is to move it into the coefficient semifield (which also means adjoining its inverse $x_{i_j}^{-1}$) and to remove it from the mutation process.

\begin{definition}
  A \defit{cluster localization} of a cluster algebra $A$ is a freezing $A[x_{i_1}^\dagger,\dots,x_{i_s}^\dagger]$ such that
  \[
    A[x_{i_1}^\dagger,\dots,x_{i_s}^\dagger] = A[x_{i_1}^{-1},\dots,x_{i_s}^{-1}].
  \]
\end{definition}

A cluster algebra is \defit{isolated} if the exchange matrix is the zero matrix.
A finite family $A_1$, \dots,~$A_l$ of cluster localizations of $A$ is a \defit{cluster cover} of $A$ if for every $P \in \Spec(A)$, there exists some $i$ with $PA_i \ne A_i$.
Geometrically, this means that $\Spec(A)$ is covered by the images of the maps $\Spec(A_i) \to \Spec(A)$ induced by the localizations.
In particular, one has $A = \bigcap_{i=1}^l A_i$ \cite[Proposition 2]{M14}

\begin{definition}[{\cite[Definition 2]{M14}}]
  A cluster algebra $A$ is \defit{locally acyclic} if it has a cluster cover by isolated cluster algebras.
\end{definition}

A seed $(\vec x, \vec y, B)$ is \defit{acyclic} if there is no sequence $i_1$, \dots,~$i_l \in [1,n]$ with $l \ge 2$ such that $b_{i_si_{s+1}}>0$ for all $s \in [1,l-1]$, and $i_1=i_l$.
A cluster algebra is \defit{acyclic} if it has an acyclic seed (acyclicity is not mutation-invariant).
Locally acyclic cluster algebras equivalently can be defined as those cluster algebras that have a cluster cover by acyclic cluster algebras.

The main results of \cite{M14} carry over from $\bZ$ to arbitrary domains $D$.

\begin{theorem}[\cite[Theorem 2]{M14}]
  If $A$ is locally acyclic, then $A=U$.
\end{theorem}

\begin{remark}
The equality $A=U$ also holds for cluster algebras arising from surfaces without punctures and one marked point on the boundary \cite{canacki-lee-schiffler15}, for cluster algebras of moduli spaces of $G$-local systems \cite{ishibashi-oya-shen23}, and symmetric Poisson nilpotent algebras \cite{goodearl-yakimov23}.
The equality is sensitive to the ground ring \cite{bucher-machacek-shapiro19}.
\end{remark}

Locally acyclic cluster algebras are, as intersections of integrally closed domains, always integrally closed.
Suppose $D$ is noetherian and $\bP$ is finitely generated. Then $D\bP$ is noetherian, and locally acyclic cluster algebras over $D$ with coefficients in $\bP$ are finitely generated as $D \bP$-algebras \cite[Theorem 4.2]{M13} (using \cite[\href{https://stacks.math.columbia.edu/tag/00EP}{Lemma 00EP}]{stacks-project}).
In particular, such cluster algebras are noetherian and integrally closed, and hence Krull domains.
Below we will extend this result to ground rings that are Krull domains (\cref{t:krull-cluster-algebras}).

\begin{remark} \label{r:no-cover-by-cluster-tori}
  In general, it is not possible to cover $\Spec(A)$ by cluster tori, that is, by localizations of the form $D\bP[\vec x^{\pm 1}]$ for clusters $\vec x$ of $A$, even when $A$ is acyclic.
  Indeed, if a cover by cluster tori exists, then $A$ is non-singular \cite[Proposition 2.4]{BM25}.
  However, even acyclic cluster algebras, such as the $A_n$-type cluster algebras, can have singularities \cite{BFMS23}.
\end{remark}

\section{Laurent Intersection Rings} \label{s:lir}

In this section we introduce (finite) Laurent intersection rings and determine their class groups (\cref{t:class-group-over-krull-domains}).
This lays the algebraic groundwork for the computational work in \cref{s:computable-domains}.
We also discuss iterated Laurent intersection rings in \cref{s:iterated-flirs} and conclude with a full characterization of the (easy) one-variable case, in \cref{s:1-var-lir}.

\begin{definition} \label{d:lir}
  Let $D$ be a domain with field of fractions $K$ and let $K(\vec x)$ with $\vec x=(x_1,\dots,x_n)$ be the field of rational functions in $n$ variables. 
  A subring $A \subseteq K(\vec x)$ is a \defit{Laurent intersection ring over $D$} \textup(or \defit{$D$-LIR}\textup) if there exists a nonempty set $S \subseteq K(\vec x)^n$ such that
  \begin{propenumerate}
  \item \label{lir:algindep}\label{lir:first} we have $K(\vec x)=K(\vec y)$ for all $\vec y \in S$,
  \item \label{lir:intersection} we have
  \[
   A = \bigcap_{\vec y \in S} D[\vec y^{\pm 1}], 
  \]
  \item \label{lir:containment}\label{lir:last} and $D[\vec y] \subseteq A$ for all $\vec y \in S$.
  \end{propenumerate} 
\end{definition}

If the base ring $D$ is clear from context, we just write that $A$ is a Laurent intersection ring or a LIR.
If $A$ is a LIR, a set $S$ satisfying properties \ref{lir:first}--\ref{lir:last} is called a \defit{system of charts}. 
An element of $S$ is a \defit{chart}.

By definition, a chart consist of $n$ algebraically independent elements over $K$ that generate the entire field $K(\vec x)$.
By a change of variables on $K(\vec x)$, we may therefore assume $\vec x \in S$ whenever convenient.
If $n=0$, then $A=D$.
We usually exclude this trivial case.

Ultimately, we are interested in Krull domains in this class of rings, where the following holds.

\begin{lemma} \label{l:lir-krull-subsystem}
  Let $D$ be a domain.
  If $A$ is $D$-LIR and a Krull domain, then every system of charts $S$ for $A$ contains a finite subsystem of charts $\{ \vec y(0)=\vec x, \vec y(1), \dots, \vec y(r) \} \subseteq S$.
  Moreover, each $P \in \htop(A)$ extends non-trivially to at least one of the localizations $D[\vec y(i)^{\pm 1}]=A[\vec y(i)^{-1}]$ with $i \in [0,r]$.
\end{lemma}

\begin{proof}
  Let $S$ be a system of charts for $A$.
  By definition, we have $A = \bigcap_{P \in \htop(A)} A_P$, where the $A_P$ are DVRs and the intersection has finite character.
  If $\vec y=(y_1,\dots,y_n) \in S$, then $D[\vec y^{\pm 1}]=A[(y_1 \cdots y_n)^{-1}]$ is a localization since $D[\vec y] \subseteq A$ by \ref{lir:containment} of \cref{d:lir}. Therefore,
  \[
  D[\vec y^{\pm 1}] = \bigcap_{\substack{P \in \htop(A) \\ y_1\cdots y_n \not \in P}} A_P,
  \]
  by \cite[Proposition 1.8]{FOSSUM}.

  Fix some $\vec x=(x_1,\dots,x_n) \in S$.
  By the finite character property of a Krull domain, there are only finitely many $P_1$, \dots,~$P_r \in \htop(A)$ with $x_1\cdots x_n \in P_i$.
  Using \ref{lir:intersection} of \cref{d:lir},  
  \[
  \bigcap_{P \in \htop(A)} A_P = A = \bigcap_{\vec y \in S} D[\vec y^{\pm 1}] = \bigcap_{\vec y \in S} \bigcap_{\substack{P \in \htop(A) \\ y_1\cdots y_n \not \in P}} A_P.
  \]
  Removal of any DVR $A_P$ in the first intersection yields a proper overring (by the Approximation Theorem for Krull domains \cite[Theorem 5.8]{FOSSUM}, there exists an $f \in K(\vec x)$ such that $\val_P(f)< 0$ and $\val_{Q}(f) \ge 0$ for all $Q \in \htop(A) \setminus \{P\}$).
  It follows that, for each $i$, there exists some $\vec y(i)=(y_1(i),\dots,y_n(i)) \in S$ such that $y_1(i) \cdots y_n(i) \not \in P_i$.
  Hence,  
  \[
  A = D[\vec x^{\pm 1}] \cap \bigcap_{i=1}^r D[\vec y(i)^{\pm 1}],
  \]
  and so $\big\{\vec x, \vec y(1), \dots, \vec y(r)\big\}$ is a finite system of charts for $A$.
\end{proof}

The previous lemma motivates the following key definition.

\begin{definition} \label{d:flir}
  Let $D$ be a domain.
  A ring $A$ is a \defit{finite Laurent intersection ring over $D$} \textup(or \defit{$D$-FLIR}\textup) if it is a $D$-LIR having a finite system of charts.
\end{definition}

\begin{proposition} \label{p:lir-krull}
  Let $A$ be a LIR over some domain $D$.
  Then $A$ is a Krull domain if and only if $D$ is a Krull domain and $A$ is a $D$-FLIR.
\end{proposition}

\begin{proof}
  Suppose first that $D$ is a Krull domain and that $A$ is a finite Laurent intersection ring, say $A= \bigcap_{\vec y \in S} D[\vec y^{\pm 1}]$ with $S \subseteq K(\vec x)^n$.
  Each $D[\vec y^{\pm 1}]$ is a Krull domain, and finite intersections of Krull domains with the same field of fractions are Krull domains, so $A$ is a Krull domain.

  Conversely, if $A$ is a Krull domain and $\vec x$ is a chart for the Laurent intersection ring $A$, then the localization $A[(x_1\cdots x_n)^{-1}]=D[\vec x^{\pm 1}]$ is a Krull domain. 
  Therefore, the ground ring $D$ is a Krull domain.
  \Cref{l:lir-krull-subsystem} shows that $A$ is a finite Laurent intersection ring.
\end{proof}

If $D \subseteq A$ is an extension of Krull domains, there is a natural homomorphism $\cF_v(D) \mapsto \cF_v(A)$ given by $I \mapsto (IA)_v$.
If the extension satisfies the property \defit{PDE} (for \emph{pas d'èclatement}), meaning $\height(P \cap D) \le 1$ for every $P \in \htop(A)$, then this induces a homomorphism $\Cl(D) \to \Cl(A)$ \cite[Theorem 6.2 and Proposition 6.3]{FOSSUM}. 
We verify that $D$-FLIRs satisfy the PDE property.

\begin{lemma}\label{l:PDE}
  Let $A$ be a $D$-FLIR over a Krull domain $D$.
  Then $\height(P\cap D)\le 1$ for every $P\in \htop(A)$.
\end{lemma}

\begin{proof}
  If $P\in \htop(A)$, there exists a chart $\vec{y}$ such that $P D[\vec{y}^{\pm1}]\ne D[\vec{y}^{\pm 1}]$.
  Since $D[\vec{y}^{\pm 1}]$ is a localization of $A$, then also $\height(P D[\vec{y}^{\pm1}])=1$.
  If there exists $0 \ne Q\in \Spec(D)$ such that $Q\subseteq P\cap D$, then
  \[
  0 \ne Q D[\vec{y}^{\pm 1}]\subseteq (P\cap D) D[\vec{y}^{\pm 1}]\subseteq P D[\vec{y}^{\pm1}],\]
  hence $Q D[\vec {y}^{\pm 1}]=P D[\vec{y}^{\pm1}]$, since the latter has height one, and so $Q = P \cap D$.
  This shows $\height(P\cap D)\le 1$.
\end{proof}

We can now determine $\Cl(A)$ in terms of $\Cl(D)$.
In a first step, we show that Nagata's Theorem leads to split short exact sequence.

\begin{lemma} \label{l:class-group-splits}
  Let $A$ be a $D$-FLIR over a Krull domain $D$ and let $\vec{x}=(x_1,\dots,x_n)$ be a chart.
  Then there is a commutative diagram
  \[
  \begin{tikzcd}
    0 \ar[r] & H \ar[r,hook] & \Cl(A) \ar[r, "\alpha"] & \Cl(D[\vec{x}^{\pm1}]) \ar[r] & 0\\
             &          &                         & \Cl(D) \ar[u, "\thicksim" {inner sep=.4mm, rotate=90, anchor=north}, "\beta"] \ar[ul, "\phi"] & 
  \end{tikzcd}
  \]
  with $\alpha([I])= \big[ID[\vec{x}^{-1}]\big]$ for $I \in \cF_v(A)$, with $\beta([J]) = \big[J D[\vec{x}^{\pm 1}]\big]$ and $\phi([J])=[(JA)_v]$ for $J \in \cF_v(D)$, and with $H$ generated by the classes of those $P\in \htop(A)$ containing $x_1\cdots x_n$.
  The top row is exact, and the map $\phi \circ \beta^{-1}$ splits the short exact sequence in the top row. In particular,
  \[  
    \Cl(A)\cong \Cl(D)\oplus H.
  \]
\end{lemma}

\begin{proof}
  Since $A[\vec x^{-1}] = D[\vec{x}^{\pm1}]$ is the localization of $A$ at the multiplicative set generated by $x_1$, \dots,~$x_n$, Nagata's Theorem (\cref{t:nagata}) yields the exact top row of the commutative diagram, with $H$ and $\alpha$ as claimed.
  
  Further, there is an isomorphism $\beta$ given by extension of ideals, that is, for $J \in \cF_v(D)$, we have $\beta([J]) = \big[J D[\vec{x}^{\pm1}]\big]$ \cite[Theorem 8.1]{FOSSUM}.

  Because of \cref{l:PDE}, extension of divisorial fractional ideals induces a homomorphism $\phi\colon \Cl(D)\to \Cl(A)$, given by $\phi([J]) = [(JA)_v]$ for $J \in \cF_v(D)$.
  We verify that $\phi \circ \beta^{-1}$ splits the short exact sequence.
  This means showing $\alpha \circ \phi = \beta$.
  Let $J \in \cF_v(D)$.
  Then
  \[
  \alpha\circ\phi([J]) = \alpha([(JA)_v]) = \big[ (JA)_v D[\vec{x}^{\pm1}] \big] = \big[ (J D[\vec{x}^{\pm1}])_v \big] = \big[ J D[\vec{x}^{\pm1}] \big] = \beta([J]).
  \]

  Since the short exact sequence splits, in particular $\Cl(A) \cong \Cl(D) \oplus H$.
\end{proof}

We can now give an explicit description of $\Cl(A)$.

\begin{theorem} \label{t:class-group-over-krull-domains}
  Let $A$ be a $D$-FLIR over a Krull domain $D$, and let $\vec{x}$ be a chart.
  Then 
  \[  
    \Cl(A)\cong \Cl(D)\oplus \bZ^r \,/\, \langle \vec c_1, \dots, \vec c_n \rangle
  \]
  where $P_1$, \dots,~$P_r$ are the height-one primes that contain $x_1\cdots x_n$ and $\vec{c}_i=(c_{ij})_{1\le j\le r} \in \bZ^{r}$ is such that $x_i A = \vprod_{1 \le j \le r } P_j^{c_{ij}}$. 
\end{theorem}

\begin{proof}
  Let $P_1$, \dots,~$P_r$ be the height-one primes of $A$ containing $x_1\cdots x_n$.
  \Cref{l:class-group-splits} shows $\Cl(A)\cong \Cl(D)\oplus H$ with $H$ generated by $[P_1]$, \dots,~$[P_r]$.
  Therefore, it is sufficient to prove 
  \[
  H\cong \bZ^r \,/\, \langle \vec c_1, \dots, \vec c_n \rangle.
  \]
  The proof is the same as in \cite[Thereom 3.1]{GELS19}.
  We include it for the sake of completeness.
  
  We need to prove that if $\sum_{j=1}^r m_{j}[P_j]=0$ in $H$ for some $m_1$, \dots,~$m_r \in \bZ$, then there exist $u_1$, \dots,~$u_n\in \bZ$ such that $m_j=\sum_{i=1}^n u_i c_{ij}$ for all $i\in[1,r]$.
  Then $(m_1,\dots,m_r) = \sum_{i=1}^n u_i \vec c_i$, and the claim follows.

  So assume $\sum_{j=1}^r m_{j}[P_j]=0$. 
  Then $\vprod_{1 \le j \le r} P_j^{m_j}$ is a fractional principal ideal of $A$, say
  \begin{equation} \label{eq:first-fact-a}
  \vprod_{1 \le j \le r} P_j^{m_j}=a A
  \end{equation}
  for some $0 \ne a \in K(\vec x)$, where $K$ is the field of fractions of $D$.
  Localizing at the multiplicative set generated by $x_1$, \dots,~$x_n$,
  \[
  aA[\vec x^{-1}] = aD[\vec{x}^{\pm 1}] = \vprod_{{\substack{1 \le j \le r  \\ x_1\cdots x_n \not \in P_j}}} (P_j D[\vec x^{\pm 1}])^{m_i}=D[\vec{x}^{\pm 1}],
  \]
  since $x_1\cdots x_n \in P_j$ for all $j\in[1,r]$ by choice of the $P_j$.
  
  We find $a\in D[\vec{x}^{\pm1}]^\times$, and hence $a=\varepsilon x_1^{u_1}\cdots x_n^{u_n}$ for some $\varepsilon\in D^\times\subseteq A^\times$ and $u_1,\ldots,u_n\in \bZ$.
  Substituting the factorizations of $x_iA$, we get
  \begin{equation} \label{eq:second-fact-a}
  aA = \vprod_{1\le i \le n}\Big(\ \vprod_{1\le j \le r}P_j^{c_{ij}}\Big)^{u_i}=\vprod_{1\le j\le r}P_j^{\sum_{i=1}^{n}u_ic_{ij}}.
  \end{equation}
  Comparing the exponents of $P_j$ in the two factorizations of $a$ into height-one primes in \cref{eq:first-fact-a} and \cref{eq:second-fact-a}, we obtain $m_j=\sum_{i=1}^{n}u_ic_{ij}$ for all $j\in[1,r]$, as desired.
\end{proof}

\begin{remark}
  The class group of a FLIR need not be torsion-free, even for factorial $D$: for example, full rank generalized upper cluster algebras are FLIRs by \cite[Theorem 3.11]{GSV18}, and their class group can also have torsion \cite[Examples 4.13]{Pom25b}.
  Similarly, we will see below that one-variable FLIRs may have torsion in their class group (\cref{exm:flir-1-var}).
  On the other hand, the class group of an upper cluster algebra that is a Krull domain is always torsion-free \cite[Theorem 4.4]{Pom25a}.
\end{remark}

To understand the factorization theory of a Krull domain globally (that is, its arithmetic invariants), it is not sufficient to know its class group, but one also needs to know how many height-one prime ideals lie in each class.
Then the multiplicative monoid of $A$ modulo $A^\times$ is uniquely determined up to isomorphism \cite[Theorem 2.5.4.3]{GH06}.

For cardinality reasons, there are at most $\min\{\card{D}, \aleph_0\}$ many height-one prime ideals in $A$.
The following shows that this upper bound is attained in each class, thereby giving a complete description of the distribution of prime divisors in the class group of a FLIR over a Krull domain.
This is analogous to the situation in Krull monoid algebras \cite{FW22}.

\begin{theorem} \label{t:many-primes}
  If $A$ is a FLIR in $n \ge 1$ variables over a Krull domain $D$, then every class of $\Cl(A)$ contains $\min\{\card{D}, \aleph_0\}$ height-one prime ideals.
\end{theorem}

To study the factorization of any given element $a \in A$, it is however not necessary to know this distribution: it is sufficient to work with the finitely many height-one prime ideals containing $a$.
We therefore postpone the somewhat lengthy proof of \cref{t:many-primes} to \cref{s:distribution}.
\subsection{Iterated FLIRs}
\label{s:iterated-flirs}
Iterated FLIRs will be useful in \cref{s:cluster-algebras} and \cref{s:final-observations} to deal with isolated cluster algebras and freezings of cluster algebras.

\begin{definition}
  Let $D$ be a domain.
  A ring $A$ is an \defit{iterated LIR} over $D$, respectively an \defit{iterated FLIR} over $D$, if there exist rings $D = A_0 \subseteq A_1 \subseteq \cdots \subseteq A_r = A$ such that each $A_i$ is an $A_{i-1}$-LIR, respectively an $A_{i-1}$-FLIR, for $i \in [1,r]$.
\end{definition}

The important observation is that iterating the LIR construction still yields a LIR over the original ground ring.

\begin{proposition}\label{p:iterated-lir}
  If $A$ is an iterated LIR over $D$, then it is a $D$-LIR.
  More specifically, if $S_i$ is a system of charts for $A_i$ over $A_{i-1}$ for each $i \in [1,r]$, then
  \[
    S = \{\, \big(\vec x(1), \vec x(2), \dots, \vec x(r)\big) : \vec x(i) \in S_i \text{ for } i \in [1,r] \,\} \cong S_1 \times \dots \times S_r.
  \]
  is a system of charts for $A$ over $D$.
\end{proposition}

\begin{proof}
  We prove the claim by induction on $r$.
  The case $r=1$ is trivial.
  So assume $r>1$ and that the claim holds for $r-1$.
  Then $A_{r-1}$ is a $D$-LIR with system of charts $T\cong S_1 \times \dots \times S_{r-1}$ and $A$ is an $A_{r-1}$-LIR with system of charts $S_r$.
  Then
  \[
  A = \bigcap_{\vec{x}\in S_r}A_{r-1}[\vec x^{\pm 1}] = \bigcap_{\vec{x}\in S_r} \Big(\bigcap_{\vec{y}\in T}D[\vec y^{\pm 1}]\Big)[\vec{x}^{\pm1}]=\bigcap_{(\vec{y},\vec{x})\in T \times S_r}D[\vec{y}^{\pm 1}, \vec{x}^{\pm 1}].
  \]
  This shows \ref{lir:intersection} of \cref{d:lir}.
  To check \ref{lir:algindep} and \ref{lir:containment} of \cref{d:lir}, let $\vec{x}\in S_r$ and $\vec{y}\in T$.
  Then
  \[
  K(\vec{x},\vec{y})=K(\vec{y})(\vec{x})\qquad\text{and}\qquad D[\vec{y},\vec{x}] = D[\vec y][\vec x] \subseteq A_{r-1}[\vec x] \subseteq A. \qedhere
  \]
\end{proof}

We obtain the following consequence, by using finite sets $S_i$.

\begin{corollary} \label{c:iterated-flir}
  If $A$ is an iterated FLIR over a domain $D$, then it is a $D$-FLIR.
\end{corollary}

Since a Laurent polynomial ring is trivially a FLIR, we also obtain the following.

\begin{corollary} \label{c:lir-over-laurent}
  Let $D[\vec t^{\pm 1}]$ be a Laurent polynomial ring in a finite number of variables over a domain $D$.
  If $A$ is a $D[\vec t^{\pm 1}]$-FLIR with system of charts $S$, then $A$ is a $D$-FLIR with system of charts
  \[
  \{\, (\vec t, \vec x) : \vec x \in S \,\}
  \]
\end{corollary}

\subsection{One-variable LIRs} \label{s:1-var-lir}

We discuss the special case of one-variable LIRs, where a complete classification is possible.
This simple case provides explicit examples of LIRs and sheds some light on the difference between LIRs and upper cluster algebras.
In addition, the classification of the one-variable case is used in \cref{s:distribution}.

\begin{proposition} \label{p:flir-1-var}
  If $A$ is a $D$-LIR over a domain $D$ in $n=1$ variables, then $A \cong D[x]$ or $A \cong D[x^{\pm 1}] \cap D[(a/x)^{\pm 1}] = D[x,a/x]$ with $0 \ne a \in D$.
\end{proposition}

In the first case, the polynomials $x$ and $x+1$ form a system of charts for $A$.
In the second case, the Laurent polynomials $x$ and $a/x$ form a system of charts.

\begin{proof}
  Fix a chart $x$, and consider another chart $y \in D(x)$.
  We first show that only the two cases
  \[
  y=a/x\qquad  \text{or}\qquad y=ux+a \qquad\text{with $a \in D$ and $u \in D^\times$}
  \]
  can occur.

  Since $y \in D[x^{\pm 1}]$ and $x \in D[y^{\pm 1}]$, we have $y = f(x)/x^m$ and $x = g(y)/y^n$ for some nonzero $f$,~$g \in D[t]$ and $m$,~$n \in \bZ$ with $t \nmid f$ and $t \nmid g$.

  Substituting the expressions into each other, we get
  \begin{equation*}
  x = \frac{g(f(x)/x^m)}{(f(x)/x^m)^n}.
  \end{equation*}
  If $m \ne 0$, then the order of vanishing at $0$ on the right-hand side is $m(n-d)$ where $d = \deg(g)$.
  Since $m(n-d)=1$, we have $(m,n-d)=(1,1)$ or $(m,n-d)=(-1,-1)$.

  We distinguish three cases.
  \begin{itemize}
  \item If $m=0$, then $y=f(x)$. From $y=f(g(y)/y^n)$ we easily find that also $n=0$, and so $x=g(y)$.
    Thus, $\deg(f)=\deg(g)=1$, and so $y=ux+a$ with $a \in D$ and $u \in D^\times$.

  \item
   If $m=1$ and $n=d+1$, then
   \[
    f(x)^{d+1} = \sum_{i=0}^d g_i f(x)^i x^{d-i} \qquad\text{with}\qquad g(t) = \sum_{i=0}^d g_i t^i.
   \]
   If $\deg(f) \ge 1$, this is impossible by comparing degrees.
   So $f$ is constant, and $y=a/x$ with $a \in D$.

   \item If $m=-1$ and $n=d-1$, then
   \[
   x^{n+1} f(x)^n = g(f(x)x).
   \]
   Since $g(0)\ne 0$, the right-hand side has order $0$ at $0$, which gives $n=-1$.
   Looking at the equation again, now yields that $f$ and $g$ are constant, hence $y=ux$ with $u \in D^\times$.
  \end{itemize}

  If there is a chart of the form $y=ux+a$ with $a \ne 0$, then $D[x^{\pm 1}] \cap D[y^{\pm 1}] = D[x]$.
  Thus, we may assume all the charts, except for $x$ itself, are of the form $y_i = a_i/x$ with $a_i \in D$.
  We may also assume $a_i \not \in D^\times$, since otherwise $D[x^{\pm 1}] = D[y_i^{\pm 1}]$.

  Now we claim that if there are two such charts $y_1$ and $y_2$ (in addition to $x$), then $a_1 \in a_2 D^\times$.
  Indeed, we have 
  \[
  y_1 = \frac{a_1}{x} = \frac{a_1}{a_2} y_2 \qquad\text{and}\qquad y_2 = \frac{a_2}{x} = \frac{a_2}{a_1} y_1.
  \]
  Since $y_1 \in D[y_2^{\pm 1}]$ and $y_2 \in D[y_1^{\pm 1}]$, it follows that $a_1/a_2$, $a_2/a_1 \in D$, hence $a_1 \in a_2 D^\times$.
  We conclude $D[y_i^{\pm 1}] = D[y_j^{\pm 1}]$ for all $i$,~$j$.
  Therefore, we have $A= D[x^{\pm 1}] \cap D[(a/x)^{\pm 1}]$ with a chart $y=a/x$ for some $a \in D$.

  Finally, to see $D[x^{\pm 1}] \cap D[(a/x)^{\pm 1}] = D[x,a/x]$, consider $f = \sum_{i=-m}^n f_i x^i \in D[x^{\pm 1}]$ with $f_i \in D$.
  Then $f \in D[(a/x)^{\pm 1}]$ if and only if $a^{i}$ divides $f_{-i}$ whenever $i > 0$.
  It follows that $f \in D[x,a/x]$.
\end{proof}

\begin{example}
  Note that $\{x,a/x\}$ is a system of charts for $A=D[x^{\pm 1}] \cap D[(a/x)^{\pm 1}]$, but if $a \not \in D^\times$, then $\{x,x/a\}$ is not!
  While $D[(a/x)^{\pm 1}]=D[(x/a)^{\pm 1}]$, the definition of a LIR requires in addition that $D[y] \subseteq A$ for each chart $y$.
  We have $a/x \in D[x^{\pm 1}]$, but $x/a \not \in D[x^{\pm 1}]$ if $a \not \in D^\times$.
\end{example}

While cluster algebras in one variable (of geometric type, without coefficients) are trivial to deal with, as only $D[x,2/x]$ (one isolated vertex) arises, FLIRs in one variable can already be non-factorial even when the base ring $D$ is factorial, as the following example shows.

\begin{example} \label{exm:flir-1-var}
  The ring $A=\bZ[x,2/x]$ is the cluster algebra arising from a single mutable variable.
  It is factorial (say, by the criterion in \cite{GELS19} for acyclic cluster algebras).

  The ring $B=\mathbb Z[x,6/x]$ is not a cluster algebra and is not factorial:
  \[
  6 = 2 \cdot 3 = x \cdot \frac{6}{x}
  \]
  are two distinct factorizations of $6$ into atoms in $B$.
  By inspecting the proof of Case 1 of \cref{l:many-irred} below, or by a short direct computation, one sees that $B$ has two height-one prime ideals lying over $(x)$, namely $(x,2)$ and $(x,3)$.
  Furthermore, since $(x) = (x,2) \cdot (x,3)$, we see $\Cl(B) \cong \bZ^2/\langle (1,1)\rangle \cong \bZ$ from \cref{t:class-group-over-krull-domains}.

  More generally, if $a=p_1^{r_1} \cdots p_t^{r_t}$ is the prime factorization of some $a \in \bZ_{\ge 2}$,
  and $C = \mathbb Z[x,a/x]$, then $C$ has $t$ height-one prime ideals lying over $(x)$, namely $(x,p_i)$ for $i \in [1,t]$.
  Considering the localization $C' \coloneqq C[x/a]=\bZ[(a/x)^{\pm 1}]$ at $a/x$ and using $xC'=aC'$, shows $(x) = (x,p_1)^{r_1} \cdots_v (x,p_t)^{r_t}$ as divisorial product.\footnote{We need not have $(x)=(x,p_1)^{r_1} \cdots (x,p_t)^{r_t}$ without taking the divisorial closure, as can be seen with a short computation for $a=4$.}
  From \cref{t:class-group-over-krull-domains}, we conclude $\Cl(C) \cong \bZ^t / \langle (r_1, \dots, r_t) \rangle \cong \bZ^{t-1} \oplus \bZ/(d)$ where $d=\gcd(r_1, \dots, r_t)$.

  The description of height-one primes can be extended to one-variable FLIRs over Krull domains $D$, see \cref{l:1-flir-primes}.
\end{example}

\section{Computable Krull Domains}\label{s:computable-domains}

In this section we show that over suitable base rings $D$, such as $\bZ$, $\bF_q$, $\bQ$, $\algc{\bQ}$, as well as polynomial rings and fields of rational functions over these rings, it is possible to compute the class group of a $D$-FLIR as well as all factorizations of a given element into atoms (up to order and associativity). 

To treat the various base rings of interest in a uniform way, we introduce the notion of a computable Krull domain.
This entails that we can represent the elements of the domain on an idealized computer (e.g., a Turing machine), and that there are algorithms to perform the basic operations as well as to compute the class group.

In computability theory, several closely related formalizations of computable structures (computable fields, computable groups, etc.) have been proposed by different authors, for instance by Fröhlich and Shepherdson \cite{froehlich-shepherdson56}, by Rabin \cite{rabin60}, and by Mal'cev \cite{malcev61,malcev62}.
A gentle introduction can be found in \cite{miller08}; a recent textbook covering the area is \cite{downey-melnikov26}.
We follow this approach, without going into the details of computability theory, as our main focus is on the algebraic aspects.

We recall that a \defit{computable field} is a field $K$ with the following properties.
As a set, it is computably enumerable: the elements of $K$ are thought of as a subset of $\bZ_{\ge 0}$, such that there is an algorithm that lists all elements $a_0$, $a_1$,~\dots\ of $K$.
More precisely, given as input $i \in \bZ_{\ge 0}$, the algorithm eventually halts with output the field element $a_i$.
Further, there are algorithms to perform addition and multiplication in $K$: for instance, there is an algorithm that, given as input $i$ and $j \in \bZ_{\ge 0}$ (representing the field elements $a_i$ and $a_j$), outputs $k$ such that $a_i + a_j = a_k$.

Negation and inversion are then computable as well by an enumeration procedure: given $i \in \bZ_{\ge 0}$, one can search for $k$ such that $a_i + a_j = 0$ (respectively, such that $a_i a_j = 1$) by successively trying all elements of $K$.\footnote{The special element $0$ can be computed by enumeration as the only solution to $a_i+a_i=a_i$. Similarly, the element $1$ is the only solution of $a_i a_i = a_i$ with $a_i + a_i \ne a_i$.} As should be obvious from this example, the notion of computability makes no claims about the efficiency of the algorithms.

A computable field is in particular finite or countable, and given any two elements $a$,~$b$ of the field, one can decide whether $a=b$ or $a \ne b$.
Obvious examples of computable fields are $\bQ$ and $\bF_q$.
Non-examples are $\bR$ and $\bC$, as these fields are uncountable.
Less obviously, fields of rational functions $K(\vec x)$ \cite[\S42]{vdwaerden50} \cite[\S52--\S59]{edwards84} and algebraic closures $\algc{K}$ \cite{rabin60} over computable fields $K$ are again computable.
Computable rings, domains, and groups can be defined similarly.

For computable Krull domains, we will additionally need to be able to compute with divisors and their classes.
A divisor is a formal $\bZ$-linear combination of prime divisors.
While divisors correspond exactly to divisorial fractional ideals (see \cref{sssec:divisors}), the representation as a formal sum lends itself better to the computational aspects.
This is why we emphasize the divisor point of view in this section. 

A divisor $e_1 P_1 + \cdots + e_1 P_t$ is represented on a computer as a tuple $\big((e_1,P_1),\ldots,(e_t,P_t)\big)$ where $e_i \in \bZ$ and the $P_i$ are prime divisors of $D$.
The only non-trivial requirement here is that we need to have a way of representing the prime divisors in such a way that, given two prime divisors $P$ and $Q$, we can decide whether $P=Q$ or $P \ne Q$.
We enforce this by requiring that the set of prime divisors of $D$ is also computably enumerable.

This does not appear to be sufficient to allow us to compute the class group or the class of a divisor.
Hence, we add a corresponding assumption.

\begin{definition}\label{d:computable-domain}
  A \defit{computable Krull domain} is a Krull domain $D$ such that $D$ is computably enumerable,
  \begin{propenumerate}[label={\textup{(C\arabic*)}},leftmargin=1cm]
    \item\label{comp: operations} addition and multiplication in $D$ are computable,
    \item\label{comp: enum-divisors} the set $\htop(D)$ of prime divisors is computably enumerable,
    \item\label{comp: valuations} given $a \in D$, it is possible to compute the principal divisor $\dv(a)$, and
    \item\label{comp: class-group-D} the class group $\Cl(D)$ and the homomorphism $\Div(D) \to \Cl(D)$ are computable.
  \end{propenumerate}
\end{definition}

The last property means that, given a divisor $E$ of $D$, we can compute its class $[E]$ in $\Cl(D)$, and that, given two classes $[E]$,~$[E']$, we can decide whether $[E]=[E']$ and can compute $[E]+[E']$.

If $D$ is a computable Krull domain, then its field of fractions $K$ is a computable field, by the usual pair construction.
Given a computable field $K$, there is in general no algorithm to decide whether a given polynomial in $K[x]$ is irreducible or not \cite{froehlich-shepherdson56}.

A field with \defit{splitting algorithm} is a computable field $K$ such that there is an algorithm to decide whether a given polynomial in $K[x]$ is irreducible or not (equivalently, by enumeration, for every polynomial a factorization into irreducible factors is computable).
Already Kronecker showed that if $K$ has a splitting algorithm, then so does the field of rational functions $K(\vec x)$.
He also showed the following (transcribed into modern terminology).

\begin{proposition}
  If $K$ is a computable field with splitting algorithm and $f \in K[\vec x]$, it is possible to compute a factorization of $f$ into irreducible factors.
\end{proposition}

\begin{proof}
  A multivariate polynomial $f \in K[\vec x]$ can be factored by factoring the univariate polynomial $f(z, z^b, \dots, z^{b^n}) \in K[z]$ if $b \ge 2$ is chosen such that $b > \deg_{x_i}(f)$ for all $i$. 
  This is based on the uniqueness of the base-$b$ representation of integers.
  See \cite[\S42]{vdwaerden50} or \cite[\S52--\S59]{edwards84}.
  Alternatively, see \cite[Chapter 2.7]{becker-weisspfenning93} for a modern treatment.
\end{proof}

This motivates the following definition.

\begin{definition}
  A computable Krull domain $D$ has a \defit{splitting algorithm} if its field of fractions does.
\end{definition}

\begin{example}
  \begin{enumerate}
  \item A field is a computable Krull domain if and only if it is a computable field.
  Here \ref{comp: enum-divisors}, \ref{comp: valuations}, and \ref{comp: class-group-D} are vacuous.
  \item If $K$ is a computable field, then $K[x]$ is always a computable domain in the sense that addition and multiplication are computable.
  However, if $K$ does not have a splitting algorithm, then $K[x]$ is not a computable Krull domain, as \ref{comp: valuations} fails (keeping in mind that we represent divisors as formal sums of prime divisors).
  \item The ring of integers $\bZ$ is a computable Krull domain.
  Further, every ring of integers $\mathcal O$ in a number field is a computable Krull domain, by standard results from algebraic number theory (Minkowski's geometry of numbers).
  \item Analogous to the number-theoretic setting, in the function field setting, holomorphy rings of function fields (of transcendence degree one) over finite fields are computable Krull domains.
  These are the intersections of all but finitely many valuation rings of the function field.
  \end{enumerate}
\end{example}

\begin{remark}
That we only ask for factorizations in $K[x]$, and not in $D[x]$, to be computable, will a posteriori be justified by \cref{t:compute-factorizations,p: polynomial-computable-domain}, which show that over a computable Krull domain with splitting algorithm it is indeed possible to compute factorizations in $D[x]$.
\end{remark}

\subsection{Computing Factorizations in Computable Krull Domains}

If a divisorial fractional ideal $I$ of $D$ is given by generators $a_1$, \dots, $a_n$, then it is clearly possible to compute its corresponding divisor $\dv(I) = \dv(a_1) + \dots + \dv(a_n)$.
The following lemma shows that the converse direction is also computable.

\begin{lemma} \label{l:compute-generators}
  If $D$ is a computable Krull domain and $E \in \Div(D)$, it is possible to decide whether $E$ is principal.
  If $E$ is principal, it is possible to compute $a \in D$ such that $\dv(a)=E$.
  If $E$ is not principal, it is possible to compute $a$,~$b \in D$ such that $\min\{\dv(a),\dv(b)\} = E$.
\end{lemma}

\begin{proof}
  The divisor $E$ is principal if and only if $[E] \in \Cl(D)$ is zero.
  This is decidable by \ref{comp: class-group-D} in \cref{d:computable-domain}.

  If $E$ is principal, an enumeration now finds $a \in D$ such that $\dv(a)=E$, using \ref{comp: valuations}. If $E$ is not principal, we know that there exist $a$,~$b \in D$ such that $\min\{\dv(a),\dv(b)\} = E$ \cite[Proposition 5.11]{FOSSUM}.
  Enumerating pairs of elements of $D$, we can again compute such a pair.
\end{proof}

\begin{remark}
  The enumeration procedures in \cref{l:compute-generators} work in principle but are of course completely impractical.
  For specific computable Krull domains of interest, one can often substitute more practical algorithms.
  For instance, for rings of integers in number fields, finding a generator of a principal ideal amounts to enumerating lattice points of small norm in a lattice associated to the divisor (i.e., fractional ideal).
\end{remark}

\begin{lemma} \label{l:compute-atoms}
  If $D$ is a computable Krull domain and $0 \ne a \in D$, it is possible to compute, up to associativity, all atoms $b \in D$ that divide $a$.
\end{lemma}

\begin{proof}
  Compute $\dv(a) = e_1 P_1 + \dots + e_m P_m$ using \ref{comp: valuations} in \cref{d:computable-domain}.
  If $b \mid a$ then $\dv(b) = k_1 P_1 +  \dots + k_m P_m$ for some $0 \le k_i \le e_i$ for all $i \in [1,m]$.
  For each of the finitely many choices of $(k_1, \dots, k_m)$, use \cref{l:compute-generators} to decide whether the divisor $k_1 P_1 + \dots + k_m P_m$ is principal.
  The minimal nonzero elements $(k_1,\dots,k_m)$ for which this is the case (with respect to the componentwise partial order) correspond to (pairwise non-associated) atoms dividing $a$.
  Using \cref{l:compute-generators} again, compute $b \in D$ such that $\dv(b) = k_1 P_1 + \dots + k_m P_m$ for these divisors.
\end{proof}

\begin{remark}
  Since $\varphi\colon \bZ^m \to \Cl(D), P_i \mapsto [P_i]$ is a group homomorphism, the set of all $(k_1,\dots,k_m)$ corresponding to elements dividing $a$ is given by 
  \[
  \big([0,e_1] \times \cdots \times [0,e_m]\big) \cap \ker(\varphi),
  \]
  and as such corresponds to the lattice points in $\bZ^m \cap \ker(\varphi)$ of a rational polytope in $\bR^m$.
\end{remark}

\begin{theorem}\label{t:compute-factorizations}
  If $D$ is a computable Krull domain and $0 \ne a \in D$, it is possible to compute all factorizations of $a$ into atoms \textup(up to order and associativity\textup). 
\end{theorem}

\begin{proof}
  By \cref{l:compute-atoms}, we can compute, up to associativity, all atoms $b_1$, \dots,~$b_n$ dividing $a$ (possibly $n=0$ if $a \in D^\times$).
  Factorizations of $a$ correspond to $n$-tuples $(t_1,\ldots,t_n)\in \bZ_{\ge 0}^n$ such that $a=\varepsilon b_1^{t_1}\cdots b_n^{t_n}$ with $\varepsilon \in D^\times$.
  Expressed in terms of divisors, finding all factorizations is equivalent to finding all $(t_1,\ldots,t_n)\in \bZ_{\ge 0}^n$ such that
  \[
  \dv(a)=t_1\dv(b_1)+\cdots+t_n\dv(b_n).
  \]
  Note that the $t_i$ are bounded from above, since $\dv(b_i) \ne 0$ for each $i$.
  Since $\dv(a)$ and $\dv(b_i)$ can be computed by \ref{comp: valuations} in \cref{d:computable-domain}, finding the solutions reduces to solving a system of linear Diophantine equations in $\bZ_{\ge 0}^n$ with finitely many solutions.
\end{proof}

We summarize the (naive) algorithms suggested by the proofs of \cref{l:compute-atoms} and \cref{t:compute-factorizations} in pseudocode in \cref{alg:factor-atoms,alg:factor-full}.\footnote{The enumeration of tuples $(k_1,\ldots,k_m)$ in lexicographic order in \cref{alg:factor-atoms} ensures that all $(l_1,\dots,l_m)$ with $l_i \le k_i$ for all $i$ have already been considered, when checking minimality of $(k_1,\ldots,k_m)$.}

\begin{algorithm}
  \caption{Naive algorithm to compute all atoms dividing $a$}\label{alg:factor-atoms}
  \begin{algorithmic}
    \Function{AtomsDividing}{$0 \ne a \in D$}
      \State $e_1 P_1 + \cdots + e_m P_m \gets \dv(a)$ \Comment{using \ref{comp: valuations}}
      \State $A \gets \emptyset$ \Comment{will contain atoms}
      \State $C \gets \emptyset$ \Comment{will contain coefficient tuples of atoms}
      \ForAll{$(k_1,\ldots,k_m)\ne (0,\dots,0)$ with $0 \le k_i \le e_i$} \Comment{in lexicographic order}
        \If{$k_1 P_1 + \cdots + k_m P_m$ is principal} \Comment{using \cref{l:compute-generators}}
          \If{there is no $(l_1,\ldots,l_m) \in C$ with $l_i \le k_i$ for all $i$}
            \State $b \gets$ generator of $k_1 P_1 + \cdots + k_m P_m$ \Comment{using \cref{l:compute-generators}}
            \State $A \gets A \cup \{b\}$
            \State $C \gets C \cup \{(k_1,\ldots,k_m)\}$
          \EndIf
        \EndIf
      \EndFor
      \State \Return $A$
    \EndFunction
  \end{algorithmic}
\end{algorithm}

\begin{algorithm}
\begin{algorithmic}
  \caption{Naive algorithm to compute all factorization of $a$ (up to associativity)}\label{alg:factor-full}
  \Function{Factorizations}{$0 \ne a \in D$}
    \State $\{b_1,\ldots,b_n\} \gets$ \Call{AtomsDividing}{$a$}
    \State $F \gets \emptyset$
    \ForAll{$(t_1,\ldots,t_n)\in \bZ_{\ge 0}^n$ with $\dv(a)=t_1 \dv(b_1) + \dots + t_n \dv(b_n)$}
      \State $\varepsilon \gets a / (b_1^{t_1} \cdots b_n^{t_n})$ \Comment{computable in $K$}
      \State $F \gets F \cup \big\{\big(\varepsilon, (b_1, t_1),\ldots,(b_n,t_n)\big)\big\}$
    \EndFor
    \State \Return $F$
  \EndFunction
\end{algorithmic}
\end{algorithm}

\subsection{(Laurent) Polynomial Rings over Computable Krull Domains}

\begin{proposition}\label{p: polynomial-computable-domain}
  If $D$ is a computable Krull domain with splitting algorithm, then so is the polynomial ring $D[\vec x]$.
\end{proposition}

\begin{proof}
  Clearly property \ref{comp: operations} in \cref{d:computable-domain} is satisfied.

  The polynomial ring $D[\vec x]$ is a Krull domain with two types of prime divisors (\cref{sssec:div-poly}): if $P \in \htop(D[\vec x])$ satisfies $D \cap P \ne 0$, then $P$ is extended from $D$, that is $P = P' D[\vec x]$ with $P' = D \cap P \in \htop(D)$.
  These prime divisors are computably enumerable by \ref{comp: enum-divisors}.
  Computationally, we simply represent them using the corresponding prime divisor $P'$ of $D$.
  In terms of the corresponding valuation, if $f = \sum_{\vec m} a_{\vec m} \vec x^{\vec m} \in D[\vec x]$, then $\val_{P}(f) = \min\{\, \val_{P \cap D}(a_{\vec m}) : \vec m  \,\}$.
  Therefore, the valuation $\val_{P}(f)$ is computable.

  If $P \cap D = 0$, then $P$ extends to the localization $K[\vec x]$ and $P = p K[\vec x] \cap D[\vec x]$ for some irreducible polynomial $p \in K[\vec x]$.
  Non-associated irreducible polynomials give rise to distinct prime divisors.
  Hence, prime divisors of this type correspond exactly to monic irreducible polynomials in $K[\vec x]$ (monic with respect to some arbitrary total order on monomials).
  Since $K$ has a splitting algorithm, the set of monic irreducible polynomials in $K[\vec x]$ is computably enumerable.
  We represent $P$ by the corresponding monic irreducible polynomial $p \in K[\vec x]$.
  In this case $\val_{P}(f)$ is simply the multiplicity of $p$ in the factorization of $f$ in $K[\vec x]$.
  Thus, also \ref{comp: enum-divisors} and \ref{comp: valuations} are satisfied for $D[\vec x]$.
 
  To prove \ref{comp: class-group-D}, recall that  $\Cl(D[\vec{x}])\cong \Cl(D)$ by \cref{c:cl-polynomial}, so it is sufficient to show: given a prime divisor $P$ of $D[\vec{x}]$, we can compute its class $[P]$ in $\Cl(D)$.
  If $P$ is extended from $D$, then this follows from \ref{comp: class-group-D} of $D$, because $[P]_{D[\vec x]} = [P \cap D]_D$.

  Now suppose $P = p K[\vec x] \cap D[\vec x]$ with $p \in K[\vec x]$ irreducible.
  Since $P$ is in fact represented by $p$, this polynomial $p$ is indeed explicitly given.
  Now
  \[
  \dv_{D[\vec x]}(p) = P + \sum_{i=1}^r e_i P_i,
  \]
  with $P_1$, \dots,~$P_r$ pairwise distinct, with $e_i \in \bZ$, and with $P_i \cap D \ne 0$ for all $i \in [1,r]$.
  The coefficients $e_i$ are computable by \ref{comp: valuations} of $D[\vec x]$, which we have already established.

  Now 
  \[
  [P]_{D[\vec x]} = -\Big[\sum_{i=1}^r e_i P_i\Big]_{D[\vec x]} = -\sum_{i=1}^r e_i [P_i \cap D]_D,
  \]
  and the last class is computable by \ref{comp: class-group-D} of $D$.
  Therefore, property~\ref{comp: class-group-D} holds for $D[\vec x]$ as well.
 
  Finally, since $K$ has a splitting algorithm, so does $K(\vec x)$.
\end{proof}

\begin{corollary}\label{c: laurent-computable-domains}
  If $D$ is a computable Krull domain with splitting algorithm, then so is the Laurent polynomial ring $D[\vec x^{\pm 1}]$.
\end{corollary}

\begin{proof}
  Clearly \ref{comp: operations} in \cref{d:computable-domain} is satisfied.
  The ring $D[\vec x^{\pm 1}]$ arises from $D[\vec x]$ by localizing at the multiplicative set generated by the prime elements $x_1$, \dots,~$x_n$.
  Hence, the prime divisors of $D[\vec x^{\pm 1}]$ correspond exactly to the prime divisors of $D[\vec x]$, with the $n$ prime divisors $(x_1)$, \dots, $(x_n)$ omitted.
  In addition, there is a canonical isomorphism $\Cl(D[\vec x]) \cong \Cl(D[\vec x^{\pm 1}])$ (\cref{t:nagata}).
  Therefore, properties~\ref{comp: enum-divisors}, \ref{comp: valuations}, and \ref{comp: class-group-D}  for $D[\vec x^{\pm 1}]$ follow immediately from the corresponding properties for $D[\vec x]$.
\end{proof}

\subsubsection{Algorithms for \textup(Laurent\textup) Polynomial Rings}

To summarize the algorithmic content of the proofs of \cref{p: polynomial-computable-domain} and \cref{c: laurent-computable-domains}, we explicitly state the algorithms to compute the divisor of an element of $K(\vec x)$ with respect to $D[\vec x]$ (\cref{alg:poly-ring-divisor}), as well as the class of a prime divisor of $D[\vec x]$ as an element of $\Cl(D)$ (\cref{alg:poly-ring-class}).
For $D[\vec x^{\pm 1}]$, the functions \textsc{LPDivisor} and \textsc{LPCLassOfPrimeDivisor} are completely analogous, with the prime divisors corresponding to monomials omitted, hence we do not state them.

Keep in mind that prime divisors $P$ of $D[\vec x]$ are of two types:
\begin{itemize}
  \item if $P$ is extended from $D$, then it is represented by the prime divisor $P \cap D$ of $D$ (which has a computational representation by \ref{comp: enum-divisors} of $D$);
  \item if $P$ localizes to $K[\vec x]$, then it is represented by an irreducible polynomial $p \in K[\vec x]$.
\end{itemize}

\begin{algorithm}
  \caption{Compute a divisor with respect to $D[\vec x]$}\label{alg:poly-ring-divisor}
  \begin{algorithmic}
    \Function{PolyRingDivisor}{$0 \ne f/g \in K(\vec x)$}
      \If{$g \ne 1$}
        \State \Return \Call{PolyRingDivisor}{$f$} $-$ \Call{PolyRingDivisor}{$g$}
      \Else
        \State $\lambda p_1^{e_1} \cdots p_r^{e_r} \gets$ factor $f$ in $K[\vec x]$ \Comment{using the splitting algorithm of $K$}
        \State $f=\sum_{\vec m} a_{\vec m} \vec x^{\vec m}$ with $a_{\vec m} \in K$
        \State $C \gets \min\{\, \dv_D(a_{\vec m}) : \vec m \,\}$ \Comment{component-wise minimum, using \ref{comp: valuations} of $D$}
        \State \Return $C + \sum_{i=1}^r e_i (p_i K[\vec x] \cap D[\vec x])$
      \EndIf
    \EndFunction
  \end{algorithmic}
\end{algorithm}
\begin{algorithm}
\caption{Compute the class of a prime divisor of $D[\vec x]$}\label{alg:poly-ring-class}
\begin{algorithmic}
  \Require $P$ is a prime divisor of $D[\vec x]$
  \Function{PolyRingClassOfPrimeDivisor}{$P$}
    \If{$P$ is extended from $D$}
      \State \Return $[P \cap D]$ \Comment{using \ref{comp: class-group-D} of $D$}
    \Else \Comment{$P$ localizes to $K[\vec x]$}
      \State $p \gets$ irreducible polynomial of $K[\vec x]$ representing $P$
      \State $P + e_1 P_1 + \cdots + e_r P_r \gets$ \Call{PolyRingDivisor}{$p$} \Comment{all $P_i \ne P$ 
      extended from $D$}
      \State \Return $-\sum_{i=1}^r e_i [P_i \cap D]$ \Comment{using \ref{comp: class-group-D} of $D$}
    \EndIf
  \EndFunction
\end{algorithmic}
\end{algorithm}

\subsection{FLIRs over Computable Krull Domains}

To be able to compute in $D$-FLIRs, the change of variables between charts need to be given by explicit Laurent polynomials.

\begin{definition} \label{d:explicit-flir}
  Let $D$ be a computable domain.
  A $D$-FLIR $A$ is \defit{explicit} if there are explicitly given $l \ge 0$ and Laurent polynomials $p_{ij}$, $q_{ij} \in K[\vec t^{\pm 1}]$ for $i \in [1,l]$ and $j \in [1,n]$ such that
  \[
  \{ \vec x, \vec y(1), \dots, \vec y(l) \}
  \]
  is a system of charts for $A$ with
  \[
  \vec y(i) = \big(p_{i1}(\vec x), \dots, p_{in}(\vec x)\big) \qquad\text{and}\qquad \vec x = \big(q_{i1}(\vec y(i)), \dots, q_{in}(\vec y(i))\big).
  \]
\end{definition}

Thus, in an explicit FLIR, we can computationally convert elements from one chart to any other by simple substitution of Laurent polynomials: the isomorphisms $K(\vec{x}) \to K(\vec{y}(i))$ are explicit.

Generalizing the results on polynomial and Laurent polynomial rings from the previous subsection, we will prove the following theorem.

\begin{theorem} \label{t:flir-computable}
  An explicit FLIR over a computable Krull domain with splitting algorithm is a computable Krull domain with splitting algorithm.
\end{theorem}

We split the proof into several lemmas.
First we need a way to represent the prime divisors of $A$.
Since $D[\vec x^{\pm 1}]=A[\vec x^{-1}]$ is the localization of $A$ at the multiplicative set generated by $x_1$, \dots,~$x_n$, the only prime divisors of $A$ that do not extend to $D[\vec x^{\pm 1}]$ are the finitely many prime divisors $P$ with $x_1\cdots x_n \in P$.
However, each such $P$ extends non-trivially to some $D[\vec y(i)^{\pm 1}]$ (\cref{l:lir-krull-subsystem}).

The following lemma shows that we can compute the prime divisors over $x_1\cdots x_n$ without repetition: when considering the chart $\vec y(i)$, we can detect and discard all those prime divisors over $x_1\cdots x_n$ that show up in some chart $\vec y(j)$ with $j < i$.

\begin{lemma} \label{l:compute-extra-prime-divisors-FLIR}
  For each $i \in [1,l]$, it is possible to compute $s_i \ge 0$ and prime divisors $Q_{i1}$, \dots,~$Q_{is_i}$ of $D[\vec y(i)^{\pm 1}]$, such that this is precisely the set of prime divisors $Q$ of $D[\vec y(i)^{\pm 1}]$ with
  \[
  x_1\cdots x_n \in Q \qquad\text{and}\qquad \vec y(j)_1 \cdots \vec y(j)_n \in Q \qquad\text{for all $j < i$}.
  \]
\end{lemma}

\begin{proof}
  Since $A$ is an explicit FLIR, we can express $x_1\cdots x_n=f(\vec y(i))$ with $f \in K[\vec t^{\pm 1}]$, and similarly $\vec y(j)_1 \cdots \vec y(j)_n = g_j(\vec y(i))$ with $g_j \in K[\vec t^{\pm 1}]$ for all $j \in [1,l]$ with $j < i$.

  Now compute $\dv_{D[\vec y(i)^{\pm 1}]}(f)$ and $\dv_{D[\vec y(i)^{\pm 1}]}(g_j)$ using \cref{c: laurent-computable-domains}.
  The desired set $\{ Q_{i1}, \dots, Q_{is_i} \}$ consists of the prime divisors that appear in the support of $\dv_{D[\vec y(i)^{\pm 1}]}(f)$ and also all the supports of $\dv_{D[\vec y(i)^{\pm 1}]}(g_j)$ for $j < i$.
\end{proof}

Thus, we can represent the set of prime divisors of $A$ as a disjoint union
\[
  \htop(A) = \big\{\, Q \cap A : Q \in \htop(D[\vec x^{\pm 1}]) \,\big\} \ \cup\  \bigcup_{i=1}^l \big\{ Q_{i1} \cap A, \dots, Q_{is_i} \cap A \big\}.
\]
Since we already know how to work with divisors in each of the localizations $D[\vec x^{\pm 1}]$, $D[\vec y(1)^{\pm 1}]$, \dots,~$D[\vec y(l)^{\pm 1}]$ of $A$, this gives us a computably enumerable representation of the prime divisors of $A$.

\begin{lemma} \label{l:computation-divisors-FLIR}
  For each $a \in K(\vec x)$, it is possible to compute the divisor $\dv_A(a)$.
\end{lemma}

\begin{proof}
  It suffices to compute $\dv_{D[\vec x^{\pm 1}]}(a)$ and $\dv_{D[\vec y(i)^{\pm 1}]}(a)$ for all $i \in [1,l]$, and to then combine the resulting information.
  This is possible by \cref{c: laurent-computable-domains}.
\end{proof}

\begin{lemma}\label{l:computation-class-group-FLIR}
  The class group $\Cl(A)$ and the homomorphism $\Div(A) \to \Cl(A)$ are computable.
\end{lemma}

\begin{proof}
  By \cref{t:class-group-over-krull-domains}, the class group of $A$ is $\Cl(D)\oplus H$, with 
  \[
    H\cong \bZ^r/\langle \vec{c}_1,\ldots,\vec{c}_n\rangle \qquad \text{ and } \qquad  \dv_A(x_i) = \sum_{j=1}^r c_{ij} P_j,
  \] 
  where $P_1$, \dots,~$P_r\in \htop(A)$ are the prime divisors of $A$ containing $x_1\cdots x_n$, and $\vec{c}_i=(c_{i1},\ldots,c_{ir})\in \bZ^r$ for all $i\in [1,n]$.
  Here $\Cl(D) \cong \Cl(D[\vec x^{\pm 1}])$ via the extension map $[I] \mapsto \big[ID[\vec x^{\pm 1}]\big]$.

  By \cref{c: laurent-computable-domains}, the homomorphism $\Div(D[\vec x^{\pm 1}]) \to \Cl(D)$ is computable.
  Since $D[\vec x^{\pm 1}]=A[\vec x^{-1}]$ is the localization of $A$ at the multiplicative set generated by $x_1$, \dots,~$x_n$, the map $\Div(A) \to \Div(D[\vec x^{\pm 1}])$ is given by omitting the prime divisors $P_1$, \dots,~$P_r$, and is therefore computable as well. 
  Thus, the map $\Div(A) \to \Cl(D)$ is computable.

  We can compute $\dv_{A}(x_i)$, and hence the vectors $\vec c_i$, by \cref{l:computation-divisors-FLIR}.
  Thus, also $H$ is computable.
  It remains to show that the map $\Div(A) \to H$ is computable.
  That is, given a prime divisor $P \in \htop(A)$, we need to show that we can compute its image in $H$.

  Every prime divisor $P \in \htop(A)$ that does not contain $x_1\cdots x_n$ extends to a prime divisor $P' \in \htop(D[\vec x^{\pm 1}])$.
  Therefore, we have a disjoint union
  \[
  \htop(A) = \{ P_1,\dots, P_r \} \cup \mathfrak X_0 \cup \mathfrak X_1,
  \]
  where $\mathfrak X_0$ consists of prime divisors of $A$ that are of the form $P_0 D[\vec x^{\pm 1}] \cap A$ for some $P_0 \in \htop(D)$, and where $\mathfrak X_1$ consists of prime divisors of the form $p K[\vec x^{\pm 1}] \cap A$ for some irreducible polynomial $p \in K[\vec x^{\pm 1}]$.
  
  Let $P \in \htop(A)$. 
  If $P=P_i$, for some $i$, then there is nothing to do, because our chosen presentation of $H$ uses $[P_1]$, \dots,~$[P_r]$ as generators.
  So without restriction $x_1\cdots x_n \not \in P$.

  Now suppose $P \in \mathfrak X_1$, that is, we have $P K[\vec x^{\pm 1}] = p K[\vec x^{\pm 1}]$ for some irreducible polynomial $p \in K[\vec x^{\pm 1}]$.
  The kernel of $\Div(A) \to \Div(K[\vec x^{\pm 1}])$ is generated by $P_1$, \dots,~$P_r$ together with $\mathfrak X_0$.
  Therefore,
  \[
  \dv_A(p) = P + \sum_{i=1}^r e_i P_i + E,
  \]
  for some $e_i \in \bZ$ and a divisor $E$ supported on $\mathfrak X_0$.
  The coefficients $e_i$ are computable by \cref{l:computation-divisors-FLIR}.
  By \cref{l:class-group-splits}, the class of $E$ is trivial in $H$.
  Thus, we have $[P]_H = -e_1 [P_1] - \cdots - e_r [P_r]$ in $H$.
  Having expressed $[P]_H$ in terms of the generators of $H$, it is therefore computable.

  Finally, consider $P \in \mathfrak X_0$, that is $P D[\vec x^{\pm 1}] = P_0 D[\vec x^{\pm 1}]$ with $P_0 = D \cap P D[\vec x^{\pm 1}] \in \htop(D)$.
  Note that this $P_0$ is computable from $P$, because prime divisors on $D[\vec x^{\pm 1}]$ that are extended from $D$ are in fact simply represented by their corresponding prime divisors of $D$.
  \Cref{l:compute-generators} allows us to compute $a$,~$b \in D$ such that $P_0=\min\{\dv_D(a),\dv_D(b)\}$.
  The corresponding divisorial fractional ideal of $D$ is then $(aD+bD)_v$, and its extension to $A$ is $(aA+bA)_v$.
  Thus, the extension of $P_0$ to $A$ is represented by the divisor $\min\{\dv_A(a),\dv_A(b)\}$.
  Hence,
  \[
  \min\{\dv_A(a),\dv_A(b)\} = P + \sum_{i=1}^r e_i P_i,
  \]
  for some $e_i \in \bZ$ that are computable by \cref{l:computation-divisors-FLIR}.
  Since the short exact sequence in \cref{l:class-group-splits} is split by the map induced by extension of divisorial ideals from $D$ to $A$, the class of $\min\{\dv_A(a),\dv_A(b)\}$ is trivial when projected onto $H$.
  We get $[P]_H = -e_1 [P_1] - \cdots - e_r [P_r]$ in $H$.
\end{proof}

We can now show that an explicit FLIR $A$ over computable Krull domain $D$ with splitting algorithm is again a computable Krull domain (with splitting algorithm).

\begin{proof}[Proof of \cref{t:flir-computable}]
  We have to verify the properties in \cref{d:computable-domain} for $A$.
  Since $A$ is given as a subring of $D[\vec x^{\pm 1}]$, property \ref{comp: operations} is clear.
  Property \ref{comp: enum-divisors} follows from our representation of $\htop(A)$ as disjoint union of $\htop(D[\vec x^{\pm 1}])$ together with an additional finite set of prime divisors lying over $x_1\cdots x_n$ (see the discussion after \cref{l:compute-extra-prime-divisors-FLIR}).
  Now \cref{l:computation-divisors-FLIR} shows \ref{comp: valuations}, and \cref{l:computation-class-group-FLIR} shows \ref{comp: class-group-D}.
  Finally, since the field of fractions of $A$ is $K(\vec x)$, it has a splitting algorithm.
\end{proof}

\subsubsection{Algorithms for FLIRs}

We again summarize the algorithmic content of \cref{t:flir-computable} in \cref{alg:flir-divisor,alg:flir-class}.
Keep in mind that all but finitely many prime divisors of $A$ extend to $D[\vec x^{\pm 1}]$, and the remaining prime divisors, which lie over $x_1\cdots x_n$, extend non-trivially to some $D[\vec y(i)^{\pm 1}]$.
In \cref{l:compute-extra-prime-divisors-FLIR} we have seen that we can enumerate the prime divisors of the latter type as $Q_{i1} \cap A$, \dots,~$Q_{is_i} \cap A$ in a non-redundant way.

\begin{algorithm}
  \caption{Compute the divisor of an element of a FLIR $A$} \label{alg:flir-divisor}
  \begin{algorithmic}
    \Function{FLIRDivisor}{$0 \ne f \in K(\vec x)$}
      \State $E \gets$ \Call{LPDivisor}{$f$} \Comment{divisor of $D[\vec x^{\pm 1}]$}
      \State \ForAll{$i \in [1,l]$}
        \State $f_i \in K(\vec y(i)) \gets$ expression of $f$ in chart $\vec y(i)$
        \State $E_i \gets$ \Call{LPDivisor}{$f_i$} \Comment{divisor of $D[\vec y(i)^{\pm 1}]$}
        \State \ForAll{$j \in [1,s_i]$}
          \State $e_{ij} \gets$ coefficient of $Q_{ij}$ in $E_i$
          \State $E \gets E + e_{ij} (Q_{ij} \cap A)$
        \EndFor
      \EndFor
      \State \Return $E$
    \EndFunction
  \end{algorithmic}
\end{algorithm}

\begin{algorithm}
\caption{Compute the class of a prime divisor of a FLIR $A$ in $\Cl(D) \oplus H$}  \label{alg:flir-class}
\begin{algorithmic}
  \Require $P$ is a prime divisor of $A$
  \Function{FLIRClassOfPrimeDivisor}{$P$}
    \State $c_1 P_1 + \cdots + c_r P_r \gets$ \Call{FLIRDivisor}{$x_1\cdots x_n$} \Comment{using 
    \cref{alg:flir-divisor}}
    \If{$P=P_i$ for some $i$} 
      \State \Return $(0,\ [P_i])$ 
    \Else \Comment{$P$ extends to $D[\vec x^{\pm 1}]$}
      \State $g \gets $ \Call{LPCLassOfPrimeDivisor}{$P D[\vec x^{\pm 1}]$} \Comment{class in $\Cl(D)$ (\cref{alg:poly-ring-class})} 
      \If{$P$ extends to $K[\vec x^{\pm 1}]$}
        \State $p \gets$ irreducible polynomial of $K[\vec x^{\pm 1}]$ representing $P$
        \State $P + e_1 P_1 + \cdots + e_r P_r + E \gets$ \Call{FLIRDivisor}{$p$}
        \State \Return $\big(g, \ -\sum_{i=1}^r e_i [P_i] \big)$
      \Else \Comment{$P D[\vec x^{\pm 1}]$ is extended from $D$}
        \State $P_0 \gets P D[\vec x^{\pm 1}] \cap D$
        \State Compute $a$,~$b \in D$ with $P_0=\min\{\dv_D(a),\dv_D(b)\}$ \Comment{using \cref{l:compute-generators}}
        \State $P + \sum_{i=1}^r e_i P_i \gets \min\{\dv_A(a),\dv_A(b)\}$ \Comment{using \cref{alg:flir-divisor}}
        \State \Return $\big( g, \ -\sum_{i=1}^r e_i [P_i] \big)$
      \EndIf
    \EndIf
  \EndFunction
\end{algorithmic}
\end{algorithm}

While we do not prove formal complexity bounds, from the proof of \cref{t:flir-computable} it is clear that the complexity of the algorithms is dominated by the complexity of the following computations:
\begin{enumerate}
  \item \label{ass-complexity:divgen} given a divisor $E$ of $D$, to compute $a$,~$b \in D$ with $\min\{\dv_D(a),\dv_D(b)\} = E$;
  \item \label{ass-complexity:factor} to factor multivariate polynomials in $K[\vec x]$.
\end{enumerate}

Problem~\ref{ass-complexity:divgen} is trivial for fields and $\bZ$, and practical algorithms exist for interesting classes of rings, such as rings of integers of number fields, where this corresponds to the computation of generators of an ideal.

Problem~\ref{ass-complexity:factor} is well-studied in computer algebra, and practical algorithms exist for a variety of fields, such as number fields and finite fields, as well as $\bZ$ \cite{kaltofen92,kaltofen03}.

Finally, while \cref{t:flir-computable} shows that $A$ is a computable Krull domain, in practice, an important piece is missing: to compute factorizations in $A$, we also need to compute generators of principal divisors of $A$ (not just of $D$).
While this is always possible in principle (\cref{l:compute-generators}), the enumeration procedure is wholly impractical.

\Cref{alg:flir-principal-gen} gives a more practical way, under the reasonable assumption that practical algorithms are available for Problems~\ref{ass-complexity:divgen} and \ref{ass-complexity:factor} above (in fact, instead of \ref{ass-complexity:divgen}, only the computation of a generator of a principal divisor of $D$ is needed).

\begin{algorithm}
  \caption{Compute a generator of a principal divisor $E$ in a FLIR $A$}\label{alg:flir-principal-gen}
  \begin{algorithmic}
    \Require $E$ is a principal divisor of $A$
    \Function{FLIRPrincipalGenerator}{$E \in \Div(A)$}
      \State $E' \gets$ extension of $E$ to $\Div(D[\vec x^{\pm 1}])$
      \State $E'' \gets$ extension of $E'$ to $\Div(K[\vec x^{\pm 1}])$
      \State $E'' = e_1 [p_1] + \cdots + e_l [p_l]$ with pairwise non-associated prime elements $p_i \in K[\vec x^{\pm 1}]$
      \State $f \gets p_1^{e_1} \cdots p_l^{e_l}$
      \State Compute $d \in D$ with $\dv_{D[\vec x^{\pm 1}]}(d) = E' - \dv_{D[\vec x^{\pm 1}]}(f)$.
      \State Solve the linear system $\dv_A(d f) + \sum_{i=1}^n m_i \dv_A(x_i) = E$ for $\vec m=(m_1,\dots,m_n)$.
      \State \Return $d \vec x^{\vec m} f$
    \EndFunction
  \end{algorithmic}
\end{algorithm}

\begin{proof}[Proof of Correctness of \cref{alg:flir-principal-gen}]
  First note that $E'$ is computed from $E$ by dropping the prime divisors $P_1$, \dots,~$P_r$ lying over $x_1\cdots x_n$.
  Similarly, the divisor $E''$ is computed from $E'$ by dropping the prime divisors of $D[\vec x^{\pm 1}]$ extended from $D$, which is trivial to do, given how we represent prime divisors.
  Now $E''$ has the stated representation, because prime divisors of $K[\vec x^{\pm 1}]$ are represented by irreducible polynomials. 
  Thus, we can compute $f$.

  Now $E' - \dv_{D[\vec x^{\pm 1}]}(f) \in \Div(D[\vec x^{\pm 1}])$ only has nonzero coefficients at prime divisors extended from $D$.
  Since $E'$ is in addition principal (because $E$ is), by assumption on $D$, we can compute $d$.

  Finally, $\dv_A(d f)$ only differs from $E$ at the prime divisors $P_1$, \dots,~$P_r$ lying over $x_1\cdots x_n$.
  Since $E$ is principal, this means that there exists some $\vec m \in \bZ^n$ such that $\dv_A(d \vec x^{\vec m} f) = E$.
  This is a linear equation in $\vec m$ over $\bZ$ and can therefore be solved to find $\vec m$.
\end{proof}

Finally, \cref{t:class-group-over-krull-domains} shows $\Cl(A) \cong \Cl(D) \oplus H$, with $H$ explicitly given as $\Coker(\vec x \mapsto \vec x C)$ for a matrix $C \in \bZ^{n \times r}$.
\Cref{alg:flir-class-group} computes this matrix $C$.
It is then easy to deduce further properties of $H$, for instance, the invariant factor decomposition is computable using the Smith normal form.

\begin{algorithm}
  \caption{Compute a presentation of $\Cl(A)/\Cl(D)$ for a FLIR $A$}\label{alg:flir-class-group}
  \begin{algorithmic}
    \Function{FLIRClassGroupPresentation}{A}
      \State $E_i \gets$ \Call{FLIRDivisor}{$x_i$} for all $i \in [1,n]$ \Comment{using \cref{alg:flir-divisor}}
      \State $E_i = \sum_{j=1}^r c_{ij} P_j$.
      \State \Return $C=(c_{ij})$
    \EndFunction
  \end{algorithmic}
\end{algorithm}

\section{Laurent Intersection Rings and Cluster Algebras}\label{s:cluster-algebras}

In this section we observe that locally acyclic cluster algebras are FLIRs (\cref{ssec:locally-acyclic-lir}).
An interesting subclass of locally acyclic cluster algebras are the Banff cluster algebras.
We show that these are explicit FLIRs, and hence that class groups and factorizations can be computed algorithmically (\cref{ssec:banff-flir}).
We also show how to compute presentations of Banff cluster algebras (\cref{ssec:fin-pres}) and provide a family of examples (\cref{ssec:examples}).

By definition, every upper cluster algebra is a LIR. 
This gives a characterization of upper cluster algebras that are Krull domains.

\begin{corollary} \label{c:krull-upper-finiteness}
  Let $D$ be a Krull domain, let $\bP$ be a semifield, and let $U$ be an upper cluster algebra over $D$ with coefficients in $\bP$.
  Then $U$ is a Krull domain if and only if it is a finite intersection of Laurent polynomial rings arising from clusters.
\end{corollary}

\begin{proof}
  Since $D$ is a Krull domain and $\bP$ is a torsion-free abelian group, the group algebra $D\bP$ is a Krull domain \cite[Theorem 15.4]{gilmer84}.
  The claim follows from \cref{l:lir-krull-subsystem}.
\end{proof}

\begin{remark} \label{r:krull-upper-cluster-algebras}
  \begin{enumerate}
  \item
  Full rank upper cluster algebras (with ground ring $D=\bZ$) are a finite intersection of Laurent polynomial rings by the Starfish Lemma \cite[Corollary 1.9]{BFZ05}.
  Since the intersection is obtained by mutating once in each direction from the initial seed, this is explicit, and hence these algebras are explicit FLIRs.
   
  \item There are Krull (upper) cluster algebras that do not have full rank (e.g., the $A_3$-cluster algebra) and for which the Starfish Lemma fails \cite[Examples 2.12]{Pom25a}.
  \cref{c:krull-upper-finiteness} implies that these cluster algebras are nevertheless a finite intersection of Laurent polynomial rings arising from clusters.

  \item Generalized upper cluster algebras \cite[Definition 2.7]{Pom25b} are LIRs.
  Hence, if they are Krull domains, they are FLIRs.
  The same is true for upper Laurent phenomenon algebras \cite[Definition 5.10]{Pom25b}.
  The factorization theory of these algebras is studied in \cite{Pom25b}.
  It is possible to view these results through the lens of LIRs, however this would not add anything substantial new, since the arguments would be similar to the existing ones.

  \item
   Let $A$ be a cluster algebra or upper cluster algebra.
   A prime ideal $P \in \Spec(A)$ is a \defit{deep point} if $x_1\cdots x_n \in P$ for each cluster $\vec{x}=(x_1,\ldots,x_n)$.
   The ideal of $A$ generated by $\{\, x_1\cdots x_n : (x_1,\dots,x_n) \text{ is a cluster} \,\}$ is the \defit{deep ideal}, and the set of all deep points is the \defit{deep locus} of $A$, see \cite{BM25} and \cite[Section 3]{CGSS24} for detailed discussions.
   Localizations of Krull domains at a prime ideal of height one are DVRs, and therefore regular.
   Hence, the deep ideal of an (upper) cluster algebra that is a Krull domain has height at least two. 
   Stated in geometric terms, the deep locus has codimension at least $2$ in the affine scheme $\Spec(A)$.

  \item If an (upper) cluster algebra is a Krull domain, then the ground ring $D$ must a Krull domain.
  Indeed, if $\vec x$ is a cluster, then $D \bP [\vec x^{\pm 1}]$ is a localization of $A$, hence a Krull domain.
  Therefore, also $D$ is a Krull domain \cite[Theorem 15.1]{gilmer84}.
  This shows that $D$ must necessarily be a Krull domain in \cref{c:krull-upper-finiteness}.
  \end{enumerate}
\end{remark}

The main known examples of Krull upper cluster algebras are full rank upper cluster algebras and locally acyclic cluster algebras.
Since the factorization theory of full rank upper cluster algebras is already well understood \cite{CKQ22,Pom25a,Pom25b}, we focus on the locally acyclic case, even though the approach via FLIRs also easily applies to the full rank case by the Starfish Lemma.

\subsection{Locally Acyclic Cluster Algebras as FLIRs} \label{ssec:locally-acyclic-lir}
The next two results are minor variations of results of Muller \cite{M14}, stated in the language of FLIRs.
Since we need the slightly more explicit statement of \cref{l:isolated-algebra} and, using the machinery of iterated FLIRs, the proofs are short, we reprove them.

\begin{lemma}[{essentially \cite[Proposition 3]{M14}}] \label{l:isolated-algebra}
  Let $A$ be an isolated cluster algebra over a domain $D$ with coefficients in $\bP$.
  If $\vec x$ is a cluster of $A$ with exchange polynomials $f_i$, then $A$ is a $D\bP$-FLIR with system of charts
  \[
  \big\{\, (y_1,\dots, y_n) : y_i \in \{ x_i, f_i/x_i \} \text{ for $i \in [1,n]$} \,\big\}
  \]
  Additionally, if $D$ and $\bP$ are computable and $\vec x$ is explicitly given, then $A$ is an explicit $D\bP$-FLIR.
\end{lemma}

\begin{proof}
  Since mutations in different directions commute for an isolated cluster algebra, the only cluster variables are $x_1$, \dots,~$x_n$ and $x_1'$, \dots,~$x_n'$, with $x_ix_i' = f_i\in D\bP$ for all $i \in [1,n]$.
  Therefore,
  \[
    A = D\bP[x_1, x_1', \ldots, x_n, x_n'] = D\bP[x_1, f_1/x_1][x_2, f_2/x_2] \cdots [x_n, f_n/x_n].
  \]
  Thus, the algebra $A$ is an iterated FLIR, and \cref{c:iterated-flir} shows that $A$ is a $D\bP$-FLIR with the stated system of charts.
\end{proof}

\begin{theorem}[{essentially \cite[Theorem 2]{M14}}] \label{t:krull-cluster-algebras}
  Every locally acyclic cluster algebra $A$ over a domain $D$ with coefficient semifield $\bP$ is a $D\bP$-FLIR.
  In particular, if $D$ is Krull domain, then so is $A$.
\end{theorem}

\begin{proof}
  If $A$ is a locally acyclic cluster algebra, then it has a cluster cover by isolated cluster localizations $A_1$, \dots,~$A_l$.
  Each such $A_i$ is a freezing of $A$, and hence a cluster algebra with coefficients in $\bP_i$ for a suitable semifield $\bP_i$ (see \cref{sssec:locally-acyclic}).
  \Cref{l:isolated-algebra} shows that $A_i$ is a $D\bP_i$-FLIR.
  Since $D\bP_i$ is a Laurent polynomial ring over $D\bP$, \cref{c:lir-over-laurent} implies that $A_i$ is also a $D\bP$-FLIR.
  Taking the union of the systems of charts of $A_1$, \dots,~$A_l$ gives a system of charts for $A$, so also $A$ is a $D\bP$-FLIR.
  
  Since $D$ is a Krull domain and $\bP$ is a torsion-free abelian group, the group algebra $D\bP$ is a Krull domain \cite[Theorem 15.4]{gilmer84}.
  \Cref{p:lir-krull} then shows that $A$ is a Krull domain.
\end{proof}

\begin{remark}
  As in Muller's original ``$A=U$'' Theorem \cite[Theorem 2]{M14} for locally acyclic cluster algebras, of course $\Spec(A)$ is covered by the spectra of $\Spec(A_1)$, \dots,~$\Spec(A_l)$.
  The claim in \cref{t:krull-cluster-algebras} is weaker, in that $A$ is an intersection of Laurent polynomial rings, but some prime ideals of height at least two may become trivial in all Laurent polynomial rings (see \cref{r:no-cover-by-cluster-tori}).
\end{remark}

\subsection{The Banff Algorithm and Explicit FLIRs} \label{ssec:banff-flir}
Muller introduced a semialgorithm, known as the Banff algorithm, that takes as input a seed $(\vec x, \vec y, B)$ and, if it terminates, produces a cluster cover of the associated cluster algebra $A=A(\vec x,\vec y,B)$ by isolated cluster localizations \cite[\S 5]{M13}.

Termination of the Banff algorithm therefore implies that $A$ is locally acyclic.
We straightforwardly adapt the Banff algorithm to show that $A$ is an explicit FLIR (\cref{d:explicit-flir}).
The Banff algorithm is based on covering pairs and \cref{t:banff-algorithm}.

\begin{definition}\label{d:covering-pair}
  Let $(\vec x, \vec y, B)$ be a seed with exchange matrix $B=(b_{ij})$.
  \begin{enumerate}
    \item A pair $(i,j)$ of indices is a \defit{covering pair} if $b_{ij} > 0$, but there does not exist an infinite sequence $(i_k)_{k\in \bZ}$ of indices such
    that $b_{i_ki_{k+1}}>0$ for all $k\in \bZ$ and $(i,j)=(i_1,i_2)$.
    \item A pair $(i,j)$ is a \defit{terminal covering pair} if $b_{ij} > 0$ and $b_{ki} \le 0$ for all $k$ or $b_{jk} \le 0$ for all $k$.
  \end{enumerate}
\end{definition}

The notion of covering pairs is generalized from the corresponding notion for quivers.
To understand it better, consider the directed graph with set of vertices $[1,n]$ and with an arrow from $i$ to $j$ whenever $b_{ij}>0$.\footnote{This makes sense because $B$ is skew-symmetrizable, so $b_{ij}>0$ if and only if $b_{ji}<0$.}
Then a covering pair is an arrow $i \to j$ that is not contained in any bi-infinite directed path.
A terminal covering pair is an arrow $i \to j$ with either $i$ a source or $j$ a sink. 

If there exists a covering pair $i\to j$, then there also exists a terminal covering pair as one can simply walk along a path containing $i \to j$ until reaching either a source (going backward from $i$) or a sink (going forward from $j$) \cite[Proposition 8.1]{M13}.
Terminal covering pairs will sometimes be easier to deal with in algorithmic considerations.

\begin{theorem}[Muller]\label{t:banff-algorithm}
  Let $D$ be a domain and $\bP$ a semifield.
  Let $A=A(\vec{x},\vec{y},B;D)$ be a cluster algebra.
  If there exists a covering pair $(i,j)$, then the following hold.

  \begin{enumerate}
  \item The spectra $\Spec(A[x_i^{-1}])$ and $\Spec(A[x_j^{-1}])$ cover $\Spec(A)$.
  \item If the freezings $A[x_i^\dag]$ and $A[x_j^\dag]$ are locally acyclic, then they are equal to $A[x_i^{-1}]$ and $A[x_j^{-1}]$, respectively.
  Thus, they are cluster localizations of $A$.
  \end{enumerate}
\end{theorem}

\begin{proof}
  The first claim is \cite[Corollary 5.4]{M13}, and the second claim follows from \cite[Lemma 1]{M14}.
  The arguments go through in our slightly more general setting.
\end{proof}

The idea behind the \defit{Banff algorithm} is to look for a seed $(\widetilde{\vec x},\widetilde{\vec y},\widetilde{B})$ mutation-equivalent to the input $(\vec x,\vec y, B)$ and having a covering pair $(i,j)$, so that $\Spec(A)$ is covered by the spectra of the two localizations of $\widetilde{x}_i$ and $\widetilde{x}_j$.
The procedure is applied recursively to the two freezings of $\widetilde{x}_i$ and $\widetilde{x}_j$, continuing until one reaches an isolated (or acyclic) cluster algebra.
If this recursive process terminates after finitely many steps, then the cluster algebra $A$ is said to be \defit{Banff}.
In this case the freezings are indeed cluster localizations, and we obtain a cluster cover of $A$ by isolated cluster localizations.

The Banff algorithm was introduced in \cite[\S5]{M13}.
Additional descriptions can be found in \cite{MS16,bucher-machacek-shapiro19}.

\begin{theorem}\label{t:banff-explicit-FLIR}
  If $D$ is a computable Krull domain and $\bP$ is a computable semifield, then every Banff cluster algebra with ground ring $D$ and coefficients in $\bP$ is an explicit $D\bP$-FLIR.
\end{theorem}

\begin{proof}
  Under the stated assumptions on $D$ and $\bP$, the Banff algorithm produces an explicit cluster cover of $A$ by isolated cluster localizations $A_i$.
  \cref{l:isolated-algebra} shows that each $A_i$ is an explicit $D\bP$-FLIR.
  The union of charts is an explicit system of charts for $A$ (as in the proof of \cref{t:krull-cluster-algebras}).
\end{proof}

Pseudocode for the Banff algorithm to compute a system of charts for a Banff cluster algebra is given in \cref{alg:charts-locally}.
Once we have a system of charts, the results of \cref{s:computable-domains} allow us to compute the class group and factorizations of elements in Banff cluster algebras.

\begin{algorithm}
  \caption{Compute a system of charts for a Banff cluster algebra.}\label{alg:charts-locally}
  \begin{algorithmic}
      \Require $A$ is a Banff cluster algebra with initial seed $(\vec{x},\vec{y},B)$.
      \Ensure The output is a system of $D\bP$-charts for $A$.
      \Function{Charts}{$\vec x, \vec y, B$} 
          \If{$B = 0$} \Comment{isolated case}
              \State $m \gets$ rank of $(\vec x, \vec y, B)$
              \ForAll{$i \in [1,m]$}
                  \State $f_i \gets \frac{y_i}{y_i\oplus 1}\prod_{b_{ki}>0}x_k^{b_{ki}}+\frac{1}{y_i\oplus 1}\prod_{b_{ki}<0}x_k^{-b_{ki}}$
              \EndFor
              \State $S\gets \emptyset$
              \ForAll{$I\subseteq [1,m]$}
                  \State $\vec x_I \gets (f_i/x_i, x_j : i \in I, j \notin I)$ \Comment{as in \cref{t:krull-cluster-algebras}}
                  \State $S \gets S \cup \{\vec x_I\}$
              \EndFor
              \State \Return $S$
          \Else
              \ForAll{seeds $(\widetilde{\vec x},\widetilde{\vec y},\widetilde{B})$ obtained by mutating $(\vec{x},\vec y,B)$}
                  \State $(i,j) \gets$ \Call{FindCoveringPair}{$\widetilde{B}$} 
                  \If{$(i,j)$ exists}
                      \State $(\vec x^\dagger, \vec y^\dagger, B^\dagger) \gets$ freezing of $(\widetilde{\vec x},\widetilde{\vec y},\widetilde{B})$ at $i$
                      \State $(\vec x^\ddagger, \vec y^\ddagger, B^\ddagger) \gets$ freezing of $(\widetilde{\vec x},\widetilde{\vec y},\widetilde{B})$ at $j$
                      \State $S^\dagger \gets \Call{Charts}{\vec x^\dagger, \vec y^\dagger, B^\dagger}$ \Comment{$D\bP^\dagger$-charts}
                      \State $S^\ddagger \gets \Call{Charts}{\vec x^\ddagger, \vec y^\ddagger, B^\ddagger}$ \Comment{$D\bP^\ddagger$-charts}
                      \State $S^\dagger$,~$S^\ddagger \gets$ extend $S^\dagger$, $S^\ddagger$ to $D\bP$-charts \Comment{using \cref{c:lir-over-laurent}}
                      \State \Return $S^\dagger \cup S^\ddagger$
                  \EndIf
              \EndFor
          \EndIf
      \EndFunction
  \end{algorithmic}
  \end{algorithm}

\begin{corollary}
  Let $D$ be a computable Krull domain with splitting algorithm, let $\bP$ be a computable semifield, and let $A$ be a Banff cluster algebra over $D$ with coefficients in $\bP$. Let $K$ be the field of fractions of $D\bP$.
  Then there are algorithms to compute:
  \begin{enumerate}
    \item a system of charts for $A$;
    \item the class group $\Cl(A)$ of $A$;
    \item for every nonzero element $f \in A$, all its factorizations into atoms \textup(up to order and associativity\textup);
    \item for every nonzero element $f \in K(\vec x)$, its divisor in $\Div(A)$ as well as the class of its divisor in $\Cl(A)$.
  \end{enumerate}
\end{corollary}

\begin{proof}
  This follows from \cref{t:banff-explicit-FLIR,t:flir-computable,t:compute-factorizations}.
\end{proof}

\begin{proof}[Proof of \cref{t:main-banff}]
  This is a special case of the previous corollary.
  Testing for factoriality means computing $\Cl(A)$ and checking whether it is trivial.
\end{proof}

\subsection{Presentations for Banff Cluster Algebras} \label{ssec:fin-pres}

Matherne and Muller first gave an algorithm to compute presentations of totally coprime upper cluster algebras \cite{MM15}.
In this section, we show how to compute a presentation for a Banff cluster algebra $A$.
We need a standard lemma from commutative algebra.

\begin{lemma}[{\cite[\href{https://stacks.math.columbia.edu/tag/00EO}{Lemma 00EO}]{stacks-project}}]\label{l: glueingprop}
  Let $R$ be a ring and let $g_1$, \dots,~$g_l\in R$ be such that $(g_1,\ldots,g_l)=R$.
  Let $f\colon M\to N$ be an $R$-module homomorphism. If $f\colon M[g_i^{-1}]\to N[g_i^{-1}]$ is an isomorphism for each $i\in [1,l],$ then so is $f\colon M \to N$. 
\end{lemma}

The Banff algorithm can be adapted to compute generators of the $D\bP$-algebra $A$.

\begin{proposition} \label{p:banff-generators}
  Let $D$ be a computable domain and $\bP$ a computable semifield.
  If $A$ is a Banff cluster algebra over $D$ with coefficients in $\bP$, then there exists an algorithm to compute a set $X$ of cluster variables that generate $A$ as $D\bP$-algebra.
\end{proposition}

\begin{proof}
  We use the recursive structure of the Banff algorithm.
  If $A$ is isolated, then an explicit generating set consisting of cluster variables is given in \cref{l:isolated-algebra}.

  Now suppose that $A$ has a covering pair $(i,j)$ in a seed $(\vec{x},\vec{y},B)$.
  Without restriction, we may assume that $(i,j)$ is a terminal covering pair.  
  By recursion, we can assume that for each $k \in \{i,j\}$, we have a generating set $X_k$ of $A[x_k^{-1}]$ as $D\bP_k$-algebra, where $\bP_k = \bP \oplus \Trop(x_k)$ and $X_k$ consists of cluster variables of $A[x_k^\dagger]=A[x_k^{-1}]$.

  It suffices to show that
  \[
  X\coloneqq X_i \cup X_j \cup \{x_1, \dots, x_n, x_i',x_j' \}
  \]
  generates $A$ as $D\bP$-algebra.
  Here $\vec x=(x_1,\dots,x_n)$ and $x_i'$,~$x_j'$ are the cluster variables obtained by mutating $x_i$, respectively $x_j$, in the seed $(\vec{x},\vec{y},B)$.

  First note $X \subseteq A$, since every cluster variable of the freezings $A[x_i^\dagger]$ and $A[x_j^\dagger]$ is also a cluster variable of $A$.
  Let $A'$ be the $D\bP$-algebra generated by $X$.
  Then $A' \subseteq A$, and we have to show equality.

  Since $(i,j)$ is a terminal covering pair, we have $b_{ki} \le 0$ for all $k$ or $b_{jk} \le 0$ for all $k$.
  Consider the first case.
  Now the exchange relation at $i$ reads
  \begin{equation} \label{eq:banff-exchange}
  x_i x_i' = c + x_j d m
  \end{equation}
  with $c$,~$d \in \bP$ and $m$ a monomial in $x_1$, \dots,~$x_{i-1}$,~$x_{i+1}$, \dots,~$x_n$.
  Noting that $c$ is invertible in $A'$ and $d$, $x_i'$,~$m \in A'$, we deduce $(x_i,x_j)_{A'} = A'$.
  In the second case, a symmetric argument again shows $(x_i,x_j)_{A'} = A'$.

  The inclusion $\iota\colon A' \hookrightarrow A$ induces inclusions of the localizations $\iota_k\coloneqq A'[x_k^{-1}] \hookrightarrow A[x_k^{-1}]$ for $k\in\{i,j\}$.
  By construction, the set $\iota_k(X)$ generates $A[x_k^{-1}]$ as $D\bP_k$-algebra.
  Since also $x_k$,~$x_k^{-1} \in \im \iota_k$, the homomorphism $\iota_k$ is surjective.
  Considering $\iota$ as $A'$-module homomorphism, \cref{l: glueingprop} yields $A'=A$.
\end{proof}

\begin{remark}
  The algorithm can be adapted to other locally acyclic cluster algebras, provided that (i) an explicit cover by isolated (or acyclic) cluster localizations is known, say $A[g_1^{-1}]$, \dots,~$A[g_l^{-1}]$, and (ii) one can assure $(g_1,\dots,g_l)_{A'}=A'$.
  In the Banff algorithm, the exchange relation \cref{eq:banff-exchange} of a terminal covering pair provided an easy way to guarantee the second condition by adding $x_1$, \dots,~$x_n$ and $x_i'$ to the generating set.
\end{remark}

Once a set of generators is known, a presentation can be computed using Gröbner bases.
Because of the use of Gröbner bases to compute an elimination ideal, we restrict to fields, but note that this is also possible over computable PIDs such as $\bZ$ \cite[Chapter 10.1]{becker-weisspfenning93}.

\begin{proposition}\label{p:presentation}
  Let $K$ be a computable field.
  Given $f_1$, \dots,~$f_m \in K[\vec{x}^{\pm1}]$, one can compute a presentation for the algebra $K[x_1,\ldots,x_n,f_1,\ldots,f_m]$.
\end{proposition}

\begin{proof}
    We introduce auxiliary variables $Y_1, \ldots, Y_n$ to represent the inverses $x_i^{-1}$. 
    Consider the polynomial ring
    \[
      R \coloneq K[X_1, \ldots, X_n, Y_1, \ldots, Y_n, Z_1, \ldots, Z_m].
    \]

    Define the $K$-algebra epimorphism $\phi\colon R \to K[x_1^{\pm1},\ldots,x_n^{\pm1},f_1,\ldots,f_m] $ on the generators by
    \[
      \phi(X_i) = x_i, \qquad \phi(Y_i) = x_i^{-1}, \qquad \phi(Z_j) = f_j.
    \]

    Let $f_j = g_j(\vec x) \vec{x}^{-\vec{e}_j}$ with $g_j \in K[\vec x]$ and $\vec{e}_j \in \bZ^n_{\geq 0}$ (we assume the $f_j$ are represented in such a way that $g_j$ and $\vec e_j$ can be determined).
    Then it is easy to see that
    \[
      \ker\phi=\big(\,X_iY_i-1,\, Z_j-g_j(\vec{X})\vec{Y}^{\vec{e}_j} : i\in [1,n],\, j\in [1,m]\,\big)=:I.
    \]

    Let $J$ be the elimination ideal of $I$ with respect to the variables $Y_1$, \dots,~$Y_n$, namely
    \[
      J \coloneq I \cap K[X_1, \ldots, X_n, Z_1, \ldots, Z_m].
    \]
    Then
    \[
      K[\vec X, \vec Z] / J \cong \phi(K[\vec x, \vec z]) = K[x_1,\dots,x_n,f_1,\ldots,f_m].
    \]

    Generators for $J$ can be computed by computing a Gröbner basis of $I$ with respect to a suitable elimination order (see any textbook on Gröbner bases, for instance \cite[Theorem 1.3.2]{CLO15} or \cite[Chapter 6.2]{becker-weisspfenning93}). 
    This yields the desired presentation.
\end{proof}

\begin{corollary}\label{c:banff-presentation}
  Let $A$ be a Banff cluster algebra over a computable field $K$. 
  Then there exists an algorithm to compute a presentation for $A$.
\end{corollary}

\begin{proof}
  \cref{p:banff-generators} allows us to compute a set of cluster variables $f_1$, \dots,~$f_m$ that generate $A$ as a $K$-algebra.
  Then \cref{p:presentation} shows that a presentation for $A$ is computable.
\end{proof}

\begin{remark}\label{r:associated-primes-computation}
  Over a computable field, in principle, \cref{t:class-group-over-krull-domains} allows a direct computation of the class group $\Cl(A)$ from a presentation.
  To do so, it suffices to compute the factorizations
  \[
  x_iA=\vprod_{1 \le j\le r} P_j^{a_{ij}} = \bigcap_{1 \le j\le r} (P_j^{a_{ij}})_v
  \]
  of the ideals $x_iA$ as a divisorial product of height-one prime ideals.
  Since the ideals $P_j$ are the associated primes of $A/x_i A$, and the exponents $a_{ij}$ are maximal with respect to $P_j^{a_{ij}} \subseteq x_i A$, these factorizations can be computed using algorithms for primary decompositions \cite[Chapter~8]{becker-weisspfenning93}.
  However, the computation of a presentation and of primary decompositions both use Gröbner bases, which tend to be computationally expensive.
\end{remark}

\begin{remark}
  Let $D$ be a computable domain.
  In practice, it may be preferable to compute a presentation over a smaller ring $R$ over which $D$ is an algebra, and then extend the resulting presentation to $D$ by scalar extension. 

  More concretely, suppose that $A_R$ is a Banff cluster algebra over a computable domain $R$, and $A_D$ is the corresponding cluster algebra over $D$ with the same initial seed, and $D$ is an $R$-algebra.
  Suppose we have a presentation
  \[
  0 \longrightarrow I \longrightarrow R[\vec{x}] \longrightarrow A_R \longrightarrow 0
  \]
  of the $R$-algebra $A_R$.
  If $D$ is flat over $R$, then tensoring with $D$ yields
  \[
  0 \longrightarrow I\otimes_R D \longrightarrow D[\vec{x}] \longrightarrow A_R\otimes_R D \longrightarrow 0.
  \]
  Flatness of $D$ over $R$ also guarantees
  \[
  A_R\otimes_R D \cong A_D,
  \]
  because $A_R$ is finite intersection of Laurent polynomial rings and finite intersections commute with scalar extension under the flatness hypothesis \cite[Proposition~I.2.6.6]{bourbaki-CA72}.

  If $K$ is a field and $D$ a $K$-algebra, then $D$ is flat over $K$.
  In particular, if $D$ is an algebra over a field, we can compute a presentation over the prime field and then extend scalars to $D$.
\end{remark}

\subsection{Some Examples}\label{ssec:examples}
In this subsection, we discuss a family of cluster algebras arising from triangulations of punctured discs with at least two marked points on the boundary, and test our algorithms on some of these algebras.
For details on cluster algebras arising from surfaces, see for instance \cite{FST08}.

Let $D_{m,p}$ be the disc with $m \ge 2$ marked points on the boundary and $p \ge 2$ punctures.
Any ideal triangulation of $D_{m,p}$ has $n=3p+m-3$ arcs \cite[Section 2]{FG07}.

We give an explicit description of an exchange matrix $B(m,p)=(b_{ij})$ associated to a specific triangulation of $D_{m,p}$.
The arcs of the triangulation will be labeled by integers in $[1,n]$.
The matrix entry $b_{ij}$ then counts the number of triangles in the triangulation in which the arcs $i$ and $j$ appear consecutively in clockwise order, minus the number of triangles in which they appear consecutively in counter-clockwise order.

\begin{figure}
  \centering
  \begin{tikzpicture}[inner sep=0.5mm, scale=.5,auto]
    \draw[darkgray, thick] (0,0) circle (2); 
    \fill[gray!20] (0,0) circle (2);
    \node (1) at (0,2) [fill=black, circle, inner sep=1pt] {};
    \node (2) at (0,-2) [fill=black, circle, inner sep=1pt] {};
    \node (3) at (0,0) [fill=black, circle, inner sep=1pt] {};
    \node (4) at (-1,0) [fill=black, circle, inner sep=1pt] {};
    \draw (1) -- (3) node[midway, right] {\tiny 4};
    \draw (2) -- (3) node[midway, right] {\tiny 5};
    \draw (4) -- (3) node[midway, yshift=0.05cm, xshift=0.05cm] {\tiny 3};
    \draw (4) -- (1) node[midway, left] {\tiny 1};
    \draw (4) -- (2) node[midway, left] {\tiny 2};
  \end{tikzpicture}
  \caption{A triangulation of $D_{2,2}$.}
  \label{f:triangulation}
\end{figure}
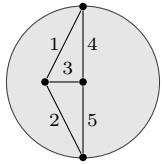

We start with a triangulation of $D_{2,2}$ as in \cref{f:triangulation}.
The corresponding exchange matrix is
\[
      B(2,2) =
      \begin{bmatrix}
        0  &  1 &  -1  & 1 &  0  \\
        -1  &  0  &  1 &  0 &  -1 \\
        1 &  -1  &  0  &  -1 & 1 \\
        -1  &  0  & 1  &  0 &  -1 \\
        0  & 1  &  -1  & 1 &  0
      \end{bmatrix}
\]

For general $(m,p)$ with $m$,~$p \ge 2$ we now recursively first construct a triangulation of $D_{m,2}$ and then one of $D_{m,p}$ with exchange matrix $B_{m,p}=(b_{ij})$.
In this construction, we always maintain that the arcs $n-1$ and $n$ appear in counter-clockwise order in a triangle that also contains a boundary segment.
The two arcs $n-1$ and $n$ will always appear in a unique triangle, so that
\[
      b_{n,n-1} = 1.
\]

\begin{figure}
  \centering
  \begin{subfigure}[t]{0.32\textwidth}
    \centering
    \begin{tikzpicture}[scale=1.3]
    \def\R{1.2}          
    \def\a{-30}          
    \def\b{30}           
  
    \coordinate (O) at (0,0);
    \coordinate (A) at (\a:\R);
    \coordinate (B) at (\b:\R);
    \coordinate (C) at (0:\R);  
  
    \draw[darkgray, thick] (0,0) circle (\R); 
    
    \fill[gray!60] (O) circle (\R);
    
    \fill[gray!30] (O) -- (A) arc (\a:\b:\R) -- cycle;
    
    \draw[thick] (O) -- (A);
    \draw[thick] (O) -- (B);
    
    \node[fill=black, circle, inner sep=1pt] at (O) {};
    \node[fill=black, circle, inner sep=1pt] at (A) {};
    \node[fill=black, circle, inner sep=1pt] at (B) {};
    
    \node[xshift=-0.5cm] at ($(O)!0.75!(A)$) {\tiny{$n$}};
    \node[xshift=-0.6cm] at ($(O)!0.8!(B)$) {\tiny{$n-1$}};
    \end{tikzpicture}
    \caption{Triangle in $D_{m,p}$ with arcs $n-1$ and $n$.}
    \label{fig:base-triangle}
  \end{subfigure}%
  \begin{subfigure}[t]{0.32\textwidth}
  \centering
  \begin{tikzpicture}[scale=1.3]
    \def\R{1.2}          
    \def\a{-30}          
    \def\b{30}           
  
    \coordinate (O) at (0,0);
    \coordinate (A) at (\a:\R);
    \coordinate (B) at (\b:\R);
    \coordinate (C) at (0:\R);  
  
    \draw[darkgray, thick] (0,0) circle (\R); 
    
    \fill[gray!60] (O) circle (\R);
    
    \fill[gray!30] (O) -- (A) arc (\a:\b:\R) -- cycle;
    
    \draw[thick] (O) -- (A);
    \draw[thick] (O) -- (B);
    \draw[thick, recursionblue] (O) -- (C);
    
    \node[fill=black, circle, inner sep=1pt] at (O) {};
    \node[fill=black, circle, inner sep=1pt] at (A) {};
    \node[fill=black, circle, inner sep=1pt] at (B) {};
    \node[fill=recursionblue, circle, inner sep=1pt] at (C) {};
    
    \node[xshift=-0.5cm] at ($(O)!0.75!(A)$) {\tiny{$n+1$}};
    \node[xshift=-0.6cm] at ($(O)!0.8!(B)$) {\tiny{$n-1$}};
    \node[recursionblue, yshift=0.2cm,xshift=0.1cm] at ($(O)!0.55!(C)$) {\tiny{$n$}};
    \end{tikzpicture}
    \caption{Adding a marked point on the boundary.}
    \label{fig:add-boundary-point}
  \end{subfigure}%
  \begin{subfigure}[t]{0.32\textwidth}
  \centering
  \begin{tikzpicture}[scale=1.3]
    \def\R{1.2}          
    \def\a{-40}          
    \def\b{40}           
  
    \coordinate (O) at (0,0);
    \coordinate (A) at (\a:\R);
    \coordinate (B) at (\b:\R);
    \coordinate (D) at (0.8,0);  
  
    \draw[darkgray, thick] (0,0) circle (\R); 
    
    \fill[gray!60] (O) circle (\R);
    
    \fill[gray!30] (O) -- (A) arc (\a:\b:\R) -- cycle;
    
    \draw[thick] (O) -- (A);
    \draw[thick] (O) -- (B);
    \draw[thick, recursionblue] (O) -- (D);
    \draw[thick, recursionblue] (A) -- (D);
    \draw[thick, recursionblue] (B) -- (D);
    
    \node[fill=black, circle, inner sep=1pt] at (O) {};
    \node[fill=black, circle, inner sep=1pt] at (A) {};
    \node[fill=black, circle, inner sep=1pt] at (B) {};
    \node[fill=recursionblue, circle, inner sep=1pt] at (D) {};
    
    \node[xshift=-0.45cm] at ($(O)!0.7!(A)$) {\tiny{$n$}};
    \node[xshift=-0.5cm] at ($(O)!0.7!(B)$) {\tiny{$n-1$}};
    \node[recursionblue, xshift=0.1cm, yshift=0.15cm] at ($(O)!0.55!(D)$) {\tiny{$n+1$}};
    \node[recursionblue, xshift=0.4cm, yshift=0cm] at ($(A)!0.55!(D)$) {\tiny{$n+3$}};
    \node[recursionblue, xshift=0.4cm, yshift=0cm] at ($(B)!0.55!(D)$) {\tiny{$n+2$}};
    \end{tikzpicture}
    \caption{Adding a puncture.}
    \label{fig:add-puncture}
  \end{subfigure}
  \caption{Recursive construction of triangulations of $D_{m+1,p}$ and $D_{m,p+1}$ from $D_{m,p}$.}
\end{figure}

We first describe how to obtain $B(m+1,2)$ from $B(m,2)$.
We start with a situation as in \cref{fig:base-triangle}, where the darker shaded region represents a part of the triangulation that plays no role in the local construction, and is omitted (by the recursive construction, the darker region does not contain a triangle with arcs $n-1$ and $n$). 

To construct $B(m+1,2)$, we add a marked point on the boundary in between the arcs $n-1$ and $n$, and an arc as in \cref{fig:add-boundary-point}. 
To maintain the counter-clockwise order of the two last arcs, we relabel the previous arc $n$ to $n+1$, and label the new arc by $n$.
We obtain a triangulation of $D_{m+1,2}$.
Denoting by $B'$ the $(n-1) \times (n-1)$ upper left submatrix of $B(m,2)$, the new $(n+1) \times (n+1)$ exchange matrix is
{
  \newcommand{\rc}[1]{\textcolor{recursionblue}{#1}}
\[
B(m+1,2) =
\begin{bNiceMatrix}
  \Block[rounded-corners,fill=lightgray]{4-4}{B'} & & & & \rc 0 & b_{1,n} \\
  & & & & \rc \vdots & \vdots \\
  & & & & \rc 0 & b_{n-2,n} \\
  & & & & \rc{-1} & \rc 0  \\
  \rc 0       & \rc \cdots  & \rc 0      & \rc 1 & \rc 0 & \rc{-1}  \\
  b_{n,1}     & \cdots  & b_{n,n-2}  & \rc 0 & \rc 1 & \rc 0   \\
\end{bNiceMatrix}.
\]
}

Now we construct $B(m,p+1)$ from $B(m,p)$.
We add a puncture inside the triangle with arcs $n-1$ and $n$, and we add three arcs as in \cref{fig:add-puncture}, obtaining a triangulation of $D_{m,p+1}$.
Denoting by $B'$ again the $(n-1) \times (n-1)$ upper left submatrix of $B(m,p)$, the new $(n+3) \times (n+3)$ exchange matrix is
{
  \newcommand{\rc}[1]{\textcolor{recursionblue}{#1}}
\[
B(m,p+1) =
\begin{bNiceMatrix}
  \Block[rounded-corners,fill=lightgray]{4-4}{B'} & & & &  b_{1,n} & \rc 0 & \rc 0 & \rc 0\\
  & & & & \vdots & \rc \vdots & \rc \vdots & \rc \vdots \\
  & & & & b_{n-2,n} & \rc 0 & \rc 0 & \rc 0\\
  & & & & \rc 0  & \rc{-1} & \rc{1} & \rc{0} \\
  b_{n,1}     & \cdots  & b_{n,n-2}  & \rc 0 & \rc 0 & \rc{1} & \rc{0} & \rc{-1} \\
  \rc 0       & \rc \cdots  & \rc 0      & \rc{1} & \rc{-1} & \rc{0} & \rc{-1} & \rc{1} \\
  \rc 0       & \rc \cdots  & \rc 0      & \rc{-1} & \rc{0} & \rc{1} & \rc 0 & \rc{-1} \\
  \rc 0       & \rc \cdots  & \rc 0      & \rc 0 & \rc{1} & \rc{-1} & \rc{1} & \rc 0 \\
\end{bNiceMatrix}.
\]
}

We have thus recursively defined a family of exchange matrices $B(m,p)$ for all integers $m \ge 2$ and $p\ge 2$ arising from triangulations of $D_{m,p}$. 
The associated cluster algebras $A_{m,p} \coloneqq A(\vec{x},\vec{y},B(m,p);D)$ are Banff cluster algebras by \cite[Theorem 10.6]{M13}.
If $p=2$, then they are acyclic cluster algebras \cite[Example 4.3]{FST08}, but if $p \ge 3$, then they are not acyclic, as a consequence of the results in \cite[\S 11 and \S 12]{FST08}.

\subsubsection{Computational Results} \label{sssec:computational-results}
We implemented two algorithms in SageMath \cite{sagemath} to compute the class group of a Banff cluster algebra.
The first method, which we will denote by \textsf{L}, is an implementation of \cref{alg:flir-class-group}, while the second method, denoted by \textsf{PD} first computes a presentation (denoted by \textsf{P}) and then primary decompositions (denoted by \textsf{D}) as explained in \cref{r:associated-primes-computation}.

We ran both algorithms on some cluster algebras of type $A(\vec{x},B_{m,p};\bQ)$ on Linux-based workstation, using version 10.7 of SageMath. 
\cref{tab:comparison} shows the running times.
As expected, method \textsf{P} is significantly slower, whereas \textsf{L} remains viable for significantly larger cases. 
For $B(2,2)$, both methods compute the class group without difficulty. However, the cluster algebras $A_{m,2}$ are in fact acyclic, so that, given a better choice of initial seed, the result is immediate from \cite{GELS19}.

For $p \ge 3$, we only managed to compute a presentation for the smallest cases $B(2,3)$ and $B(3,3)$ in less than one hour each.
Even in these cases, it was not feasible to determine the rank of the class group with \textsf{D}, since the computation of the primary decomposition exceeded the available time and memory resources.
On the other hand, method \textsf{L} performed much better, computing the rank of the class group for both $B(2,3)$ and $B(3,3)$ in a few seconds.
For all remaining cases listed in \cref{tab:comparison}, method \textsf{PD} failed already at the level of computing a presentation.
This is presumably related to the growth of the number of generators in the presentation, making the Gröbner basis computations computationally prohibitive.

In contrast, method \textsf{L} remains applicable to substantially larger examples.
Although its running time also increases with the size of the input, it avoids the need for Gr\"obner basis and is therefore more scalable in practice.

\begin{table}
  \centering
  \begin{tabular}{l c c c c c c c}
  \toprule
  & $B(2,2)$ & $B(2,3)$ & $B(3,3)$ & $B(4,3)$ & $B(3,4)$ & $B(4,4)$ & $B(5,4)$ \\
  \midrule
  rank & 11 & 2 & 2 & 2 & 3 & 3 & 3\\
  vars & 5 & 8 & 9 & 10 & 12 & 13 & 14\\
  gens & 5 & 18 & 20 & 24 & 41 & 45 & 49\\
  charts & 16 & 10 & 11 & 11 & 14 & 15 & 16\\
  \textsf{L} & 0.8s & 2.4s & 3.8s & 6.3s & 1m45s & 4m47s & 16m20s\\
  \textsf{P} & 0.1s & 4m55s & 42m30s & $*$ & $*$ & $*$ & $*$\\
  \textsf{PD} & 2.0s & $*$ & $*$ & -- & -- & -- & --\\
  \bottomrule
  \end{tabular}
  \caption{Running times of two methods for computing the divisor class group of cluster algebras, discussed in \cref{sssec:computational-results}. 
  \emph{Rank} is the rank of the class group (which is isomorphic to $\bZ^r$ for some $r$), 
  \emph{vars} is the number of cluster variables in a given seed, 
  \emph{gens} is the number of generators in the presentation used by method~\textsf{P} (that is, the parameter $m$ in \cref{p:presentation}), 
  and \emph{charts} is the number of FLIR charts used by method~\textsf{LP}.
  An asterisk indicates that the computation exceeded the time limit of one hour, while a dash indicates that step \textsf{D} was not attempted as already \textsf{P} timed out.
  }
  \label{tab:comparison}
\end{table}

Finally, we present some examples of non-unique factorizations which were computed with the methods developed in this paper.

\begin{examples}
  \begin{propenumerate}
    \item Consider the cluster algebra $A=A(\vec{x},B;\bQ(i))$ over $\bQ(i)$ associated to the matrix 
    \[
      B=\begin{bmatrix}
        0 & 0 & 0 & 2 & 0 & 0 & 0 & 0\\
        0 & 0 & 0 & -1 & 1 & 0 & 0 & 0 \\
        0 & 0 & 0 & 0 & -1 & 0 & 0 & 0\\
        -2 & 1 & 0 & 0 & 0 & 1 & -1 & 0\\
        0 & -1 & 1 & 0 & 0 & 0 & 1 & -1\\
        0 & 0 & 0 & -1 & 0 & 0 &0 &0\\
        0 & 0 & 0 & 1 & -1 & 0 & 0 &0 \\
        0 & 0 & 0 &0 & 1 &0 &0 &0\\
      \end{bmatrix}.
    \]
   The pair $(3,5)$ is a terminal covering pair for $B$.
   It is easy to check that freezing at $x_5$, we get an acyclic cluster algebra, and freezing at $x_3$, we obtain an acyclic cluster algebra after mutating at $5$.
   Therefore, the algebra $A$ is a Banff cluster algebra.
   The element
   \[f = \frac{x_4^3x_5+x_4^2x_5^2+x_4^3+x_4^2x_5+x_4x_5+x_5^2+x_4+x_5}{x_2}\]
   has four distinct factorizations in $A$:
   \begin{align*}
    f&=(x_1)\cdot (x_3)\cdot \left(\frac{x_4^2+1}{x_1}\right)\cdot \left(\frac{x_4+x_5}{x_2}\right) \cdot \left(\frac{x_5+1}{x_3}\right) \\
     &= (x_1) \cdot \left(\frac{x_4^2+1}{x_1}\right) \cdot \left(\frac{x_4+x_5}{x_2}\right) \cdot (x_8)\cdot  \left(\frac{x_5+1}{x_8}\right) \\
     &= (x_3) \cdot (x_4+i) \cdot (x_4-i) \cdot \left(\frac{x_5+1}{x_3}\right) \cdot \left(\frac{x_4+x_5}{x_2}\right) \\
      &= (x_8) \cdot \left(\frac{x_4+x_5}{x_2}\right) \cdot   \left(\frac{x_5+1}{x_8}\right) \cdot (x_4+i) \cdot (x_4-i) .
  \end{align*} 
  \item  Let $A=A(\vec{x},B(2,2);\bQ)$ be the cluster algebra associated to the triangulation of $D_{2,2}$ in \cref{f:triangulation}.
  The class group of $A$ is $\bZ^{11}$.
  Let $\vec{y}=\mu_{3}\circ \mu_4\circ \mu_5\circ\mu_1\circ \mu_2\circ \mu_3(\vec{x})$ and
  let $\vec{z}=\mu_{4}\circ\mu_3\circ\mu_5\circ\mu_2\circ\mu_1(\vec{y})$.
  Then there are 20 non-associated atoms of $A$ that divide $f$, and $f$ has 29 distinct factorizations in $A$.
  The length set of $f$ is \[\mathsf{L}(f):=\{3,4\},\] where the length of a factorization is the number of atoms in it (counted with repetition).
  As an example, here are two factorizations of $f$ of lengths $4$ and $3$, respectively: 
  {\scriptsize
  \begin{align*}
    f = & \left(\frac{x_2x_4 + x_3}{x_1}\right) \cdot  
    \left(\frac{x_2^2x_4^2 
    + 2x_1x_2x_4x_5
    + x_1^2x_5^2 
    + 2x_2x_3x_4 
    + 2x_1x_3x_5
     + x3^2} {x_1x_2x_3x_4x_5}\right)\cdot  \\ 
    & \cdot \left(\frac{
      x_{2}^{3} x_{4}^{3} + 
      x_{1} x_{2}^{2} x_{4}^{2} x_{5} + 
      3 x_{2}^{2} x_{3} x_{4}^{2} + 
      3 x_{1} x_{2} x_{3} x_{4} x_{5} + 
      x_{1}^{2} x_{3} x_{5}^{2} + 
      3 x_{2} x_{3}^{2} x_{4} + 
      2 x_{1} x_{3}^{2} x_{5} + 
      x_{3}^{3}}{x_1^2x_2 x_3x_4 x_5}   \right)\cdot (x_1) \\
    & =\left( \frac{ x_{2} x_{4}+  x_{3}}{x_1}\right) \cdot
        \left( x_{3} \right) \cdot \Bigl(\bigl(x_1^{-2}x_2^{-2}x_3^{-3}x_4^{-2}x_5^{-2}\bigr)
       \bigl( x_2^5x_4^5
      + 3x_1x_2^4x_4^4x_5
      + 3x_1^2x_2^3x_4^3x_5^2 + \\
      & \qquad + x_1^3x_2^2x_4^2x_5^3 
      + 5x_2^4x_3x_4^4
       + 13x_1x_2^3x_3x_4^3x_5
      + 12x_1^2x_2^2x_3x_4^2x_5^2
      + 5x_1^3x_2x_3x_4x_5^3
      + x_1^4x_3x_5^4 +\\
      & \qquad + 10x_2^3x_3^2x_4^3
       + 21x_1x_2^2x_3^2x_4^2x_5
      + 15x_1^2x_2x_3^2x_4x_5^2
      + 4x_1^3x_3^2x_5^3
      + 10x_2^2x_3^3x_4^2 + \\
      & \qquad + 15x_1x_2x_3^3x_4x_5
      +  6x_1^2x_3^3x_5^2
      + 5x_2x_3^4x_4
      + 4x_1x_3^4x_5 
      +  x_3^5 \bigr) \Bigr).
  \end{align*}}
  \end{propenumerate}
\end{examples}

\section{Distribution of Prime Divisors in FLIRs}
\label{s:distribution}

In this section we prove \cref{t:many-primes}.
The proof is an adaptation of the argument for cluster algebras over fields \cite[Theorem 3.2]{GELS19} and of Fadinger and Windisch's proof of an analogous statement for Krull monoid algebras \cite{FW22}.
In particular, we use lemmas of Chang (\cite[Theorem 3(ii)]{C11} or \cite[Lemma 3.2]{FW22}) and an extension of a lemma of Fadinger and Windisch \cite[Lemma 3.3]{FW22}.

The following is an adaptation of \cite[Lemma 3.3]{FW22}.

\begin{lemma} \label{l:goodgens}
  Let $D$ be a Krull domain with field of fractions $K$ and $\card{\htop(D)} = \infty$.
  Let $P_1$, \dots,~$P_t \in \htop(D)$ be pairwise distinct.
  If $0 \ne I \subseteq D$ is a divisorial ideal of $D$, then there exist infinitely many $P \in \htop(D)$ with the following property: there exist $a$,~$b \in K^\times$ such that $I^{-1}=(a,b)_v$, such that $\val_{P_i}(a) = 0$ for all $i \in [1,t]$, and such that $\val_P(a) = \val_P(b) + 1$.
\end{lemma}

\begin{proof}
  By the Approximation Theorem for Krull domains \cite[Theorem 5.8]{FOSSUM}, there exists $b \in K^\times$ such that $\val_{P}(b) = -\val_{P}(I)$ for all $P \in \htop(D)$ with $P\in \{P_1,\dots,P_t\}$ or $\val_P(I) \ne 0$, and $\val_Q(b) \ge 0$ for all remaining $Q \in \htop(D)$.
  By the same theorem, there exists $a \in K^\times$ such that $\val_{P_i}(a) = 0$ for all $i \in [1,t]$, such that $\val_{Q}(a) = -\val_{Q}(I)$ whenever $\val_{Q}(b) \ne \val_{Q}(I)$, and such that $\val_P(a) = \val_P(b) + 1$ for some fixed choice $P$ of the remaining primes, and $\val_Q(a) \ge 0$ for all other $Q \in \htop(D)$.
  Comparing valuations shows $I^{-1} = (a,b)_v$.
\end{proof}

\begin{lemma} \label{l:many-irred}
  Let $A$ be a FLIR in $n \ge 1$ variables over a Krull domain $D$ with field of fractions $K$.
  Let $\vec x = (x_1, \dots, x_n)$ be a chart of $A$, and let $P_1$, \dots,~$P_t$ be the height-one prime ideals of $A$ over $x_1\cdots x_n$.
  Then, for every divisorial ideal $0 \ne I \subseteq D$ and all $e_1$, \dots,~$e_t \in \bZ$, there exist $\min\{\card{D},\aleph_0\}$ pairwise non-associated prime elements $g \in K[\vec x^{\pm 1}]$ such that 
  \[
  \big[ gK[\vec x^{\pm 1}] \cap D[\vec x^{\pm 1}] \big] = \big[ ID[\vec x^{\pm 1}] \big] \qquad\text{and}\qquad \val_{P_i}(g) = e_i \qquad\text{ for all } i \in [1,t],
  \]
  where the first equality is in $\Cl(D[\vec x^{\pm 1}])\cong \Cl(D)$.
\end{lemma}

\begin{proof} 
  We consider several (overlapping) cases.
  Generally speaking, since we aim to produce many prime ideals, the cases when $n=1$ or when $D$ does not have infinitely many height-one primes are the harder ones.

  \medskip
  \noindent
  \textbf{Case 1: $n=1$ and $D$ is factorial.}
  Either $A=D[x]$ or $A=D[x^{\pm 1}] \cap D[(a/x)^{\pm 1}]$ for some nonzero nonunit $a \in D$ by \cref{p:flir-1-var}.
  In any case, the ring $D[x^{\pm 1}]$ is factorial, so the first property is always vacuously satisfied.

  If $A=D[x]$ or $A=D[x^{\pm 1}]$ (corresponding to $a \in D^\times$), then the second property follows easily from the fact that $D[x]$ contains $\min\{\card{D}, \aleph_0\}$ many prime elements. (Note $t=1$ and $P_1=(x)$ if $A=D[x]$, so that $x^{e_1} g$ for any non-monomial irreducible $g \in K[x]$ satisfies the property, and $t=0$ if $A=D[x^{\pm 1}]$).

  So suppose $A=D[x^{\pm 1}] \cap D[(a/x)^{\pm 1}]$ with $0 \ne a \in D \setminus D^\times$.
  Since every $P \in \htop(A)$ extends non-trivially to one of the two localizations $D[x^{\pm 1}]$ or $D[(a/x)^{\pm 1}]$, every height-one prime ideal over $x$ extends non-trivially to $D[(a/x)^{\pm 1}]$.
  There, we have
  \[
  x = a (x/a) = \pi_1^{r_1} \cdots \pi_t^{r_t} (x/a),
  \]
  with $a=\pi_1^{r_1} \cdots \pi_t^{r_t}$ a prime factorization of $a$ in $D$ and $x/a$ a unit.
  Since each $\pi_i$ is prime in $D[(a/x)^{\pm 1}]$, the ideals $P_i = \pi_i D[(a/x)^{\pm 1}] \cap A$ for $i \in [1,t]$ are precisely the height-one prime ideals of $A$ over $x$.
  Therefore, we have $\val_{P_i}(\pi_i) = 1$ and $\val_{P_j}(\pi_i) =0$ for all $i \ne j \in [1,t]$.
  Hence, we have $\val_{P_i}|_D = \val_{\pi_i}$.

  Now let $b=\pi_1^{e_1} \cdots \pi_t^{e_t} \in K^\times$.
  For sufficiently large $N$, and arbitrary $c \in D$, the polynomial
  \[
  g_c \coloneqq \pi_1^{e_1-1} x^N  + \pi_1^{e_1}c x^{N-1} + b
  \]
  satisfies $\val_{P_i}(g_c) = \val_{P_i}(b) = \val_{\pi_i}(b) = e_i$ for all $i \in [1,t]$, since $\val_{P_i}(x) > 0$ for all $i$.
  Further, it is irreducible over $D_{(\pi_1)}$ by Eisenstein's criterion.
  Hence, it is a prime element of $K[x^{\pm 1}]$ with the desired valuations at $P_i$.
  This settles \textbf{Case 1}.

  \medskip
  For the remaining cases, choose $M \ge 0$ such that $e_i \ge -M \val_{P_i}(x_1\cdots x_n)$ for all $i$ (this is possible because $\val_{P_i}(x_1\cdots x_n) \ge 1$).
  Using the Approximation Theorem for Krull domains, choose $f \in A$ such that $\val_{P_i}(f) = e_i + M \val_{P_i}(x_1\cdots x_n) \ge 0$ for all $i \in [1,t]$.
  Write
  \[
  f = \sum_{i=r}^s f_i x_1^i \in D[\vec x^{\pm 1}],
  \]
  with $f_i \in D[x_2^{\pm 1},\dots,x_n^{\pm 1}]$ and $f_r$,~$f_s \ne 0$.

  \medskip
  \noindent
  \textbf{Case 2: $n \ge 2$ and $D$ is factorial.}
  Then $D[\vec x^{\pm 1}]$ is factorial, so that the first property is always vacuously satisfied.
  Let $p \in K[x_2^{\pm 1},\dots,x_n^{\pm 1}]$ be a prime element that does not divide any of the $f_i$ and such that $\val_{P_i}(p) = 0$ for all $i$. 
  There exist at least $\min\{\card{K},\aleph_0\}$ such prime elements, because $n \ge 2$.
  Define
  \[
  g_p \coloneqq (x_1\cdots x_n)^N + f p (x_1\cdots x_n)^{-M},
  \]
  with $N$ sufficiently large so that $\val_{P_i}(g_p) = \val_{P_i}(f p(x_1\cdots x_n)^{-M}) = e_i$.
  Viewed as a Laurent polynomial in $x_1$ over $K[x_2^{\pm 1}, \dots, x_n^{\pm 1}]$, the element $g_p$ is irreducible by Eisenstein's criterion.
  Since its leading term is not divisible by a non-constant element of $K[x_2^{\pm 1}, \dots, x_n^{\pm 1}]$, the Laurent polynomial $g_p$ is a prime element of $K[\vec x^{\pm 1}]$ with the desired properties.
  Different choices of $p$ lead to non-associated prime elements $g_p$.

  \medskip
  Since a Krull domain with $\card{\htop(D)} < \infty$ is always factorial (in fact, it is a semilocal PID \cite[Proposition 2.10.7.4]{GH06}), it suffices to consider the following remaining case.

  \medskip
  \noindent
  \textbf{Case 3: $n \ge 1$ and $\card{\htop(D)}=\infty$.}  
  \Cref{l:goodgens} shows that there exist $a$,~$b \in K^\times$ and $P \in \htop(D)$ such that $I^{-1}=(a,b)_v$, that $\val_P(a) = \val_{P}(b) + 1$, and that $\val_{P_i}(a)=0$ for all $i \in [1,t]$ (we use that $P_i \cap D$ has height at most one, see \cref{l:PDE}).
  In fact, there are infinitely many choices for $P$, which allows us to also guarantee $\val_P(f_i)=0$ for all nonzero coefficients of the $f_i$.
  
  Define
  \[
  g_p \coloneqq b (x_1\cdots x_n)^N + a (x_1\cdots x_n)^{N-1} + a f (x_1\cdots x_n)^{-M} ,
  \]
  with $N$ chosen sufficiently large so that $\val_{P_i}(g_p) = \val_{P_i}(af(x_1\cdots x_n)^{-M}) = e_i$. 
  Applying Eisenstein's criterion in the localization $D_P$, it follows that $g_p \in D_P[\vec x^{\pm 1}]$ is irreducible.
  Thus, the element $g_p$ is a prime element in $K[\vec x^{\pm 1}]$.

  Finally, by \cite[Lemma 3.2]{FW22}, and noting that all coefficients of $g_p$ are multiples of $a$ or $b$, we have
  \[
  g_p K[\vec x^{\pm 1}] \cap D[\vec x^{\pm 1}] = g_p (a,b)^{-1}[\vec x^{\pm 1}] = g_p I D[\vec x^{\pm 1}]. \qedhere
  \]
\end{proof}

\bigskip
\begin{proof}[Proof of \cref{t:many-primes}]
  Fix a chart $\vec x \in S$ and let $P_1$, \dots,~$P_t$ be the height-one prime ideals of $A$ over $x_1\cdots x_n$.
  By \cref{l:class-group-splits},
  \[
  \Cl(A) \cong \Cl(D) \oplus H,
  \]
  with $H$ generated by $[P_1]$, \dots,~$[P_t]$.
  Let $\pi \colon \Cl(A) \to H$ be the projection map.

  Let $I \subseteq D$ be a nonzero divisorial ideal and let $e_1$, \dots,~$e_t \in \bZ$.
  It suffices to show: 
  there exist $\min\{\card{D},\aleph_0\}$ many pairwise distinct $Q \in \htop(A)$ such that
  \[
  \big[Q D[\vec x^{\pm 1}]\big]=\big[I D[\vec x^{\pm 1}]\big] \qquad\text{and}\qquad \pi([Q]) = e_1 \pi([P_1]) + \cdots + e_t \pi([P_t]).
  \]
  
  \Cref{l:many-irred} shows that there exist $\min\{\card{D},\aleph_0\}$ many pairwise non-associated prime elements $g \in K[\vec x^{\pm 1}]$ such that
  \[
  \big[ gK[\vec x^{\pm 1}] \cap D[\vec x^{\pm 1}] \big] = \big[ I D[\vec x^{\pm 1}]  \big] \qquad\text{and}\qquad \val_{P_i}(g) = -e_i.
  \]
  For each such $g$, the ideal $Q_g \coloneqq g K[\vec x^{\pm 1}] \cap A$ is a height-one prime ideal of $A$ with $\big[Q_g D[\vec x^{\pm 1}]\big] = \big[I D[\vec x^{\pm 1}]\big]$.

  By considering the localization $K[\vec x^{\pm 1}]$ of $A$, we obtain
  \[
  gA = Q_g \cdot_v P_1^{-e_1} \cdots \cdot_v P_t^{-e_t} \cdot_v P_1' \cdots_v P_s',
  \]
  where $P_1'$, \dots,~$P_s' \in \htop(A)$ are such that $x_1\cdots x_n \not \in P_i'$ and $D \cap P_i' \ne 0$ for all $i \in [1,s]$.
  Thus, each $P_i' D[\vec x^{\pm 1}]$ is extended from $D$.
  The splitting of the short exact sequence in \cref{l:class-group-splits} shows $\pi([P_i'])=0$ for $i \in [1,s]$.
  We conclude $\pi([Q_g]) = e_1 \pi([P_1]) + \cdots + e_t \pi([P_t])$.
\end{proof}

\section{Final Observations} \label{s:final-observations}

We make some final observations: we discuss the relation between the Picard group and the class group, revisit the acyclic case to simplify the argument of \cite{GELS19}, and close with some open questions.

\subsection{Picard Groups} \label{ssec:picard-group}
The Picard group $\Pic(A)$ of a domain $A$ is the group of all invertible fractional ideals of $A$ modulo the subgroup of principal fractional ideals.
For any Krull domain $A$, there is a canonical monomorphism $\Pic(A) \to \Cl(A)$.
For a noetherian integrally closed domain, this map is bijective if and only if $A$ is locally factorial \cite[Corollary 18.5]{FOSSUM}.

Singularities of cluster algebras have been studied, and even classified for cluster algebras of finite type \cite{BMRS15,BFMS23,BFMS24}.
For instance, the $A_3$-cluster algebra over a field of characteristic zero has a unique singular maximal ideal.
While even a non-regular local ring may be factorial, this is not the case here: the following shows that $\Pic(A)$ differs from $\Cl(A)$ for the $A_3$-cluster algebra.

\begin{proposition}
  If $K$ is a field with $\chr K=0$ and $A$ is the $A_3$-cluster algebra over $K$, then $\Pic(A) = 0$ and $\Cl(A)\cong \bZ$.
\end{proposition}

\begin{proof}
  By \cite[Corollary 5.16]{GELS19}, we have $\Cl(A)\cong \bZ$.

  By \cite[Lemma 4.1]{BFMS23}, there is a presentation
  \[
  A \cong D[Z_1,Z_2,Z_3,Z_4]/(Z_1 Z_2 Z_3 Z_4 - Z_1 Z_2 - Z_1 Z_4 - Z_3 Z_4).
  \]
  Since $\chr K \ne 2$, there is a unique singularity at the maximal ideal $M$ corresponding to the origin \cite[Theorem A(1)]{BFMS23}.

  By \cite[Corollary 18.6]{FOSSUM}, there is a short exact sequence
  \[
  0 \to \Pic(A) \to \Cl(A) \to \Cl(A_M) \to 0.
  \]
  The claim follows if we show that $\Cl(A) \to \Cl(A_M)$ is an isomorphism.

  The localization $A_M$ is isomorphic to the localization of the hypersurface 
  \[
  B\coloneqq K[X,Y,Z,W] / (XY-ZW)
  \]
  at its singular maximal ideal ${M'}=(x,y,z,w)$ \cite[Proposition 3.8]{BFMS23}.

  An adaptation of a standard argument (see \cite[Exercise II.6.5]{hartshorne77} or \cite[Proposition 14.8]{FOSSUM}) to the local case shows $\Cl(B_{M'}) \cong \bZ$: 
  localizing $B_{M'}$ further, at the height-one prime ideal $P\coloneqq (x, z)B_{M'}$, yields $B_{(x, z)} \cong K[X,Y,W]_{(X)}$, which is factorial.
  By Nagata's Theorem (\cref{t:nagata}), the class group $\Cl(B_{M'})$ is generated by the class of $P$.
  It suffices to show that no symbolic power $P^{(n)}$ with $n\ge 1$ is principal.
  Suppose $P^{(n)} = f B_{M'}$ for some $f\in B_{M'}$.
  Note $x^n$, $z^n \in P^{(n)}$.
  Since
  \[
    B_{M'} / (y, w) B_{M'} \cong K[X,Z]_{(X,Z)},
  \]
  we see that $(x^n, z^n)$ is not contained in any height-one prime ideal of $B_{M'} / (y, w) B_{M'}$, contradicting that it is contained in $fB_{M'} / (y,w) B_{M'}$.
\end{proof}

\subsection{The Acyclic Case Revisited} \label{ssec:acyclic-primes}

For a cluster algebra $A$ with acyclic seed $(\vec x, \vec y, B)$, the height-one prime ideals over $x_1\cdots x_n$ were described in \cite[Lemma 4.5, Proposition 4.7, and Theorem 4.9]{GELS19}.
Since the remaining height-one prime ideals are in bijection with those of $A[\vec x^{-1}]=D\bP[\vec x^{\pm 1}]
$, this gives a complete description of $\htop(A)$ in the acyclic case.

The argument in \cite{GELS19} is based on a rather technical induction argument and knowledge of a presentation of an acyclic cluster algebra.
Additionally, the determined generators are only proven to generate the ideals as divisorial ideals, with the question of generation as ordinary ideals left open \cite[Question 4.8]{GELS19}.
We clean up this situation in the following, by giving a more natural proof based on iterated FLIRs.
While the arguments are similar, they make more effective use of cluster localizations to avoid the technical induction.
We also allow the ground ring to be an arbitrary Krull domain.

In this subsection, let $A$ be a cluster algebra with an acyclic seed $(\vec x, \vec y, B)$ over a Krull domain $D$.
Let $f_1$, \dots,~$f_n$ be the exchange polynomials associated to the seed.
Let $x_i'=f_i/x_i$.
We generalize the notion of partners \cite[Definition 2.6]{GELS19} to the present setting.

\begin{definition}
  \begin{enumerate}
  \item
  Let $P \in \htop(D\bP[\vec x])$.
  The set of \defit{$P$-partners}, denoted by $\Partner(P) \subseteq [1,n]$, consists of all $i \in [1,n]$ for which $f_i \in P$.

  \item Two indices $i$,~$j \in [1,n]$ are \defit{partners} if they are $P$-partners for some $P \in \htop(D\bP[\vec x])$.
  \end{enumerate}
\end{definition}

From the definition of exchange polynomials, one sees that an exchange polynomial of a non-isolated index $i$ is never contained in any height-one prime ideal extended from $D\bP$ (\cref{sssec:seed-and-mutations}).
Hence, the ideal $P$ in the previous definition corresponds to an irreducible polynomial $p \in K[\vec x]$ over the field of fractions $K$ of $D\bP$.
It follows that two indices $i$,~$j$ are $P$-partners if and only if $p$ divides both $f_i$ and $f_j$ in $K[\vec x]$.

\begin{lemma} \label{l:partners-not-neighbors}
  If $i$,~$j \in [1,n]$ are partners, then $b_{ij}=b_{ji}=0$.
\end{lemma}

\begin{proof}
  In the case of characteristic $0$ and geometric coefficients this is \cite[Lemma 2.7]{GELS19}.
  If $i$ is isolated, then $b_{ij}=0$ for all $j$ by definition.
  In the non-isolated case, the argument is based on analyzing the irreducible factors of exchange polynomials, and it goes through in the general case.
\end{proof}

The following generalizes \cite[Lemma 4.5 and Proposition 4.7]{GELS19} and answers \cite[Question 4.8]{GELS19}.

\begin{proposition} \label{p:acyclic-primes}
  Let $(\vec x, \vec y, B)$ be an acyclic seed of a cluster algebra $A$ over a Krull domain $D$.
  For each $i \in [1,n]$, the height-one prime ideals of $A$ containing $x_i$ are the ideals of the form
  \[
    Q = P A + \sum_{j \in T} x_j A + \sum_{j \in \Partner(P) \setminus T} x_j' A,
  \]
  with $P \in \htop(D\bP[\vec x])$ containing the exchange polynomial $f_i$ and $\{i \} \subseteq T \subseteq \Partner(P)$.
\end{proposition}

In particular, there are
\[
  \sum_{\substack{P \in \htop(D\bP[\vec x]) \\ f_i \in P}} 2^{\card{\Partner(P)} - 1}
\]
height-one prime ideals of $A$ containing $x_i$.

We first deal with FLIRs in one variable, and then with isolated cluster algebras, by treating them as iterated FLIRs.
A height-one prime ideal of a one-variable FLIR $A=D[x,a/x]$ over a Krull domain $D$ either extends to $K[x^{\pm 1}]$, with $K$ the field of fractions of $D$, or intersects $D$ non-trivially in a height-one prime ideal by \cref{l:PDE}.
The height-one primes of the second type are explicitly described in the following.

\begin{lemma} \label{l:1-flir-primes}
  Let $D$ be a Krull domain, let $0 \ne a \in D$, and let $A=D[x,a/x]$.
  Let $P \in \htop(D)$.
  \begin{enumerate}
  \item \label{1-flir-primes:split} If $a \in P$, then
  \[
    P A + x A \quad \text{and} \quad P A + (a/x) A
  \]
  are the height-one prime ideals of $A$ over $P A$.
  \item \label{1-flir-primes:inert} If $a \not \in P$, then $P A$ is a height-one prime ideal of $A$.
  \end{enumerate}
\end{lemma}

\begin{proof}
  We have $A = D[x^{\pm 1}] \cap D[(a/x)^{\pm 1}]$ by \cref{p:flir-1-var} and every height-one prime ideal of $A$ extends non-trivially to at least one of the two localizations.

  \ref{1-flir-primes:split}
  Note that $A \cong D[x,y]/(xy - a)$ since $x$ and $a$ are coprime in $D[x]$ \cite[Lemma 14.1]{FOSSUM}.
  Thus,
  \[
  A / (P A + x A) \cong D[x,y]/(P, x, xy - a) \cong D/P[y],
  \]
  and so $PA + x A$ is a prime ideal of $A$.
  Since $PA+xA$ localizes to $P D[(a/x)^{\pm 1}] + x D[(a/x)^{\pm 1}] = P D[(a/x)^{\pm 1}]$, it has height one.
  Similarly, the ideal $P A + (a/x) A$ is a height-one prime ideal of $A$.

  The two ideals $PA + xA$ and $PA + (a/x)A$ are distinct since their localizations in the two charts are different.
  In addition, these are the only two height-one prime ideals of $A$ over $P$, since the Laurent polynomial rings $D[x^{\pm 1}]$ and $D[(a/x)^{\pm 1}]$ have no further height-one primes over $P$.

  \ref{1-flir-primes:inert}
  Now $A/PA \cong D[x,y]/(P, xy - a) \cong (D/P)[x,y]/(xy - \overline a)$ with $\overline a$ the image of $a$ in $D/P$.
  The latter ring is a domain because $x$ and $\overline{a}$ are coprime in $(D/P)[x]$ \cite[Lemma 14.1]{FOSSUM}.
  Thus, the ideal $PA$ is a prime ideal of $A$.
  Since it localizes to $P D[x^{\pm 1}]$, it has height one.
\end{proof}

\begin{lemma} \label{l:isolated-case-primes}
  The claim of \cref{p:acyclic-primes} holds for isolated cluster algebras.
\end{lemma}

\begin{proof}
  As shown in the proof of \cref{l:isolated-algebra}, an isolated cluster algebra $A$ is an iterated FLIR of the form
  \[
    D\bP[x_1, f_1/x_1][x_2, f_2/x_2] \cdots [x_n, f_n/x_n].
  \]
  By inductive application of \cref{l:1-flir-primes}, we obtain that the stated ideals are distinct height-one prime ideals of $A$ containing $x_i$.

  Conversely, suppose $Q \in \htop(A)$ contains $x_i$.
  Then $Q$ contains $f_i=x_i x_i' \in D\bP$, and hence $P \coloneqq Q \cap D\bP$ has height one by \cref{l:PDE}.
  Again by iterated application of \cref{l:1-flir-primes}, we obtain that $Q$ is of the stated form.
\end{proof}

For extending this result to acyclic cluster algebras, we use that an acyclic cluster algebra has a cluster cover by isolated cluster algebras in which only initial variables are frozen \cite[Proposition 4]{M14}. Thus, no mutation is necessary to obtain this cluster cover.

\begin{proof}[Proof of \cref{p:acyclic-primes}]
  Let $I$ be the set of all partners of $i$.
  Freezing $\{\, x_i : i \in [1,n] \setminus I \,\}$ yields an isolated cluster algebra $A^\dagger = D\bP^\dagger$ with seed $(\vec x^\dagger, \vec y^\dagger, B^\dagger)$ by \cref{l:partners-not-neighbors}.
  By \cite[Lemma 1]{M14}, the freezing $A^\dagger$ is a cluster localization of $A$.

  Note that upon localizing to $A^\dagger$, the partners of $i$ do not change.
  Therefore, \cref{l:isolated-case-primes} shows that the stated ideals are all distinct in $A^\dagger$, and hence also in $A$.

  It therefore remains to show that every height-one prime ideal $Q \in \htop(A)$ containing $x_i$ is of the stated form.
  Define $P' \coloneqq Q A^\dagger \cap D \bP^\dagger$.
  Since $f_i=x_i x_i' \in Q$, also $f_i \in P'$.
  \Cref{l:1-flir-primes} shows that $P'$ has height one (keeping in mind that $A^\dagger$ is a $D \bP^\dagger$-FLIR). 
  Now $P'$ first restricts to $P'' \coloneqq P' \cap D\bP[\vec x^\dagger]$, and then extends to the height-one prime ideal $P \coloneqq P'' D\bP[\vec x]$ of $D \bP[\vec x]$.
  Note that $P \subseteq Q$, since
  \[
  P'' = Q A^\dagger \cap D \bP[\vec x^\dagger] = (Q A^\dagger \cap A) \cap D \bP[\vec x^\dagger] = Q \cap D \bP[\vec x^\dagger].
  \]

  For each $j \in \Partner(P)$, we have $f_j=x_j x_j' \in P \subseteq Q$, so $x_j \in Q$ or $x_j' \in Q$.
  Let $\{i\} \subseteq T \subseteq \Partner(P)$ be such that $j \in T$ if and only if $x_j \in Q$.
  Define
  \[
  Q_0 \coloneqq P A + \sum_{j \in T} x_j A + \sum_{j \in \Partner(P) \setminus T} x_j' A.
  \]
  Then $Q_0 \subseteq Q$, and we claim equality.

  Since the seed $(\vec x, \vec y, B)$ is acyclic, the algebra $A$ has a cluster cover by isolated cluster algebras, where moreover each of the cluster localizations arises by freezing a subset of variables of the initial seed.
  This can be seen from the inductive construction in the proof of \cite[Proposition 4]{M14}.
  To show $Q_0=Q$, it therefore suffices to show $Q_0 A^\ddagger = Q A^\ddagger$ for every isolated cluster localization (using \cite[\href{https://stacks.math.columbia.edu/tag/00EO}{Lemma 00EO}]{stacks-project}).

  Fix a subset $I \subseteq [1,n]$ such that
  \[
  A^\ddagger\coloneqq A[x_{i}^\ddagger : i \in I] = A[x_i^{-1} : i \in I]
  \]
  is an isolated cluster algebra.
  The coefficient semifield of $A^\ddagger$ is $\bP^\ddagger \coloneqq \bP \oplus \Trop(x_i : i \in I)$, and the variables and coefficients corresponding to $i \in I$ have been removed from the seed $(\vec x^\ddagger, \vec y^\ddagger, B^\ddagger)$.
  
  Note that $Q A^\ddagger=A^\ddagger$ if and only if $T \cap I \ne \emptyset$.
  Then also $Q_0 A^\ddagger=A^\ddagger$.

  Now let $Q A^\ddagger$ be a proper ideal of $A^\ddagger$.
  Then $Q A^\ddagger$ is a height-one prime ideal of the isolated cluster algebra $A^\ddagger$.
  \Cref{l:isolated-case-primes} shows $Q A^\ddagger=Q_0 A^\ddagger$ (the partner set $\Partner(P) \setminus I$ may become smaller in $A^\ddagger$, but every generator of $Q A^\ddagger$ is present in $Q_0$).
\end{proof} 

\begin{example}
  A subtle point in the previous proof is the way of associating $P$ to $Q$.
  Passing to an isolated cluster localization may appear convoluted, but is indeed necessary: the more obvious approach of considering $Q \cap D[\vec x]$ does not work in general. Indeed, the prime ideal $Q \cap D[\vec x]$ may not even have height one in $D[\vec x]$.
  
  To illustrate this, consider the $A_3$-cluster algebra $A$ (over $\bZ$, without coefficients).
  The algebra $A$ is generated by $x_1$, $x_2$, $x_3$, $x_1'$, $x_2'$, $x_3'$, subject to the relations
  \[
    x_1 x_1' = x_3 x_3' = 1 + x_2 \qquad\text{and}\qquad x_2 x_2' = x_1 + x_3.
  \]
  Here $1$ and $3$ are $(1+x_2)$-partners.
  A height-one prime ideal of $A$ is $Q\coloneqq (x_1,x_3)=(x_1,x_3,1+x_2)$ (\cref{p:acyclic-primes}).
  The contraction $Q \cap \bZ[x_1,x_2,x_3]$ is the maximal ideal $(x_1,x_3,1+x_2)$, which has height $3$.
  On the other hand, freezing at $x_2$ (here $\{2\}$ is the complement of the partner set $\{1,3\}$), yields an isolated cluster algebra with coefficient ring $\bZ[x_2^{\pm 1}]$.
  Following the proof, now $P'=Q A[x_2^{-1}] \cap \bZ[x_2^{\pm 1}] = (1+x_2)$ has height one, and so does $P = (1+x_2) \bZ[x_1,x_2,x_3]$.
\end{example}

\subsection{Open Questions}

We close with two open questions.

\begin{question} \label{q:lir-flir}
  Is every LIR over a Krull domain a FLIR?
\end{question}

We currently do not know an example of a LIR that is not a FLIR.
In the one-variable case every LIR is a FLIR by \cref{p:flir-1-var}.
However, this case is too special to deduce a general conjecture.
A positive answer to \cref{q:lir-flir} would imply that every upper cluster algebra is a FLIR, and hence a Krull domain.
Such a result would already be interesting for special ground rings, such as fields or $\bZ$.

Currently, there is no known example of a non-Krull upper cluster algebra (under the obviously necessary condition that the ground ring is Krull).
A counterexample to \cref{q:lir-flir} may be easier to find, because the class of LIRs is much larger than the class of upper cluster algebras, already in the one-variable case (\cref{exm:flir-1-var}).

For cluster algebras, the following is known.
\begin{itemize}
\item The Markov cluster algebra (over $\bZ$) is not a Krull domain \cite[\S6]{GELS19}.
However, the upper cluster algebra is factorial and hence a Krull domain.

\item By an example of Speyer, upper cluster algebras need not be finitely generated \cite{speyer13}.
 Speyer's example is a FLIR, so FLIRs need not be finitely generated.
\end{itemize}

The second question concerns the computational complexity of the Banff algorithm.
Given an explicit FLIR, the complexity of our algorithms is clearly governed by multivariate polynomial factorization over the field of fractions of the ground ring.

If we work with locally acyclic cluster algebras, more typically, the input will be a matrix with the Banff property. 
The Banff algorithm then performs a sequence of mutations and freezings to obtain a cluster cover by isolated cluster algebras.
It is not clear how to obtain a reasonable upper bound on the complexity in this setting.

\begin{question}
  What is the computational complexity of the Banff algorithm \textup(assuming the input is a Banff exchange matrix, to assure termination\textup)?
\end{question}

Since it is not understood which exchange matrices have the Banff property, one should presumably deal with this question in natural subclasses, such as cluster algebras arising from surfaces.

\begin{adjustwidth}{-1cm}{-1cm}
\begin{singlespace}
  \printbibliography
\end{singlespace}
\end{adjustwidth}

\end{document}